\documentclass[journal]{IEEEtran}
\usepackage{bm}
\usepackage{epsfig}
\usepackage{graphicx}
\usepackage{subfigure}
\usepackage{float}
\usepackage{cite}
\usepackage{url}
\usepackage{color}
\usepackage{balance}
\usepackage{mdwlist}
\usepackage{multirow}
\usepackage{threeparttable}
\usepackage{enumitem}
\usepackage{amsmath}
\usepackage{stmaryrd}
\usepackage{booktabs}
\usepackage{siunitx}
\usepackage{amsthm}
\usepackage{mathdots}
\usepackage[numbers,sort&compress]{natbib}
\usepackage{epsfig,amsmath,amsfonts}
\usepackage{amsfonts}
\usepackage{adjustbox}
\usepackage{mathtools}
\usepackage{comment}

\makeatletter
\newif\if@restonecol
\makeatother

\usepackage[ruled,vlined]{algorithm2e}
\usepackage{algpseudocode}
\usepackage{amsmath}

\def\R{\mathbb{R}}
\def\E{\mathbb{E}}
\def\P{\mathbb{P}}

\def\S{\mathbb{S}}
\def\TN{\mathcal{TN}}
\newtheorem{theorem}{Theorem}

\newtheorem{definition}{Definition}

\newcommand{\argmin}{\mathop{\mathrm{argmin}}}

\newcommand{\st}{\mathop{\mathrm{s.t.}}}

\newcommand{\ie}{i.e.}

\newcommand{\parm}{{\xi}}
\newcommand{\vecpar}{\boldsymbol{\parm}}

\newcommand{\polyInd}{\alpha}
\newcommand{\basisInd}{\boldsymbol{\polyInd}}

\newcommand{\mat}[1]{\mathbf{#1}}

\RequirePackage{url}

\def\BibTeX{{\rm B\kern-.05em{\sc i\kern-.025em b}\kern-.08em
    T\kern-.1667em\lower.7ex\hbox{E}\kern-.125emX}}
\begin{document}
\title{PoBO: A Polynomial Bounding Method for Chance-Constrained Yield-Aware Optimization of Photonic ICs}

\author{\IEEEauthorblockN{Zichang He and Zheng Zhang, \textit{Member, IEEE}}
\thanks{This work was partly supported by NSF grants \#1763699 and \#1846476. 

Zichang He and Zheng Zhang are with Department of Electrical and Computer Engineering, University of California, Santa Barbara, CA 93106, USA (e-mails: zichanghe@ucsb.edu;  zhengzhang@ece.ucsb.edu).
}
}


\maketitle
\begin{abstract}
Conventional yield optimization algorithms try to maximize the success rate of a circuit under process variations. These methods often obtain a high yield but reach a design performance that is far from the optimal value. This paper investigates an alternative yield-aware optimization for photonic ICs: we will optimize the circuit design performance while ensuring a high yield requirement. 
This problem was recently formulated as a chance-constrained optimization, and the chance constraint was converted to a stronger constraint with statistical moments.  Such a conversion reduces the feasible set and sometimes leads to an over-conservative design. To address this fundamental challenge, this paper proposes a carefully designed polynomial function, called optimal polynomial kinship function, to bound the chance constraint more accurately.
We modify existing kinship functions via relaxing the independence and convexity requirements, which fits our more general uncertainty modeling and tightens the bounding functions.
The proposed method enables a global optimum search for the design variables via polynomial optimization. 
We validate this method with a synthetic function and two photonic IC design benchmarks, showing that our method can obtain better design performance while meeting a pre-specified yield requirement. Many other advanced problems of yield-aware optimization and more general safety-critical design/control can be solved based on this work in the future.

\end{abstract}

\section{Introduction}
\label{introduction}
The increasing process variations have resulted in significant performance degradation and yield loss in semiconductor chip design and fabrications~\cite{gielen2008emerging,chen2013process}.
Compared with electronic ICs, photonic ICs are more sensitive to
process variations (e.g., geometric uncertainties) due to their large device dimensions compared with the small wavelength. Therefore, yield modeling and optimization for photonic ICs are highly desired~\cite{lipka2016systematic,bogaerts2019layout}.

The yield optimization and yield-aware robust design have been investigated in the electronic design automation community for a long time and have been paid increasing attention in the photonic design automation~\cite{antreich1994circuit,gong2014variability,lu2017performance,pond2017predicting}.
Typical yield-aware design techniques include geometric approaches~\cite{xu2009regular}, geostatistics-motivated performance modeling~\cite{yu2007yield}, yield-aware Pareto surface~\cite{tiwary2006generation}, yield-driven iterative robust design~\cite{li2009yield}, computational intelligence assisted approaches~\cite{liu2011efficient}, corner-based method~\cite{barros2010analog}, Bayesian yield optimization~\cite{wang2017efficient} and so forth.
Yield estimation is the key component in a yield optimizer. 
Advanced yield estimators can be generally classified as Monte-Carlo-based~\cite{singhee2008practical,papoulis2001probability} and non-Monte-Carlo-based~\cite{gu2008efficient,gong2012fast,9428031} methods. Among the non-Monte-Carlo ones, surrogate modeling aims to approximate some circuit behaviors under variations to speed up the sampling and simulation process~\cite{shi2019meta,yao2014efficient,li2006asymptotic}. 
Typical surrogate models include posynomial models~\cite{li2004robust}, linear quadratic models~\cite{li2008quadratic}, support vector machine~\cite{ciccazzo2015svm,ma2020support}, Gaussian process~\cite{sanabria2020gaussian,wang2017yield}, sparse polynomial~\cite{wang2014enabling}, generalized polynomial chaos expansions~\cite{xiu2003modeling,zhang2013stochastic}, and some variants~\cite{trinchero2020combining}. 
Focusing on generalized polynomial chaos expansion, advanced techniques have also been developed to handle high-dimensional~\cite{li2010finding,zhang2014enabling,zhang2016big,He2020EPEPS,he2021high}, mixed-integer~\cite{He2019ICCAD} or non-Gaussian correlated~\cite{cui2018stochastic} process variations.
The polynomial-based modeling and optimization has been widely used in both electronics~\cite{manfredi2014stochastic,ahadi2016sparse,kaintura2018review,wang2016re, tao2018graph} and photonic IC design~\cite{waqas2021performance, waqas2018stochastic, weng2015uncertainty, weng2017stochastic}.

While most existing yield optimization approaches try to maximize the yield of a circuit, the obtained design performance (e.g., signal gain, power dissipation) may be far from the achievable optimal solution. Recently, an alternative approach was proposed in~\cite{cui2020chance} to achieve both excellent yield and design performance. Instead of simply maximizing the yield, the work~\cite{cui2020chance} optimizes a design performance metric while enforcing a high yield requirement. Specifically, the yield requirement is formulated as some chance constraints~\cite{shapiro2014lectures}, which are further transformed to tractable constraints of the first and second statistical moments. The chance-constrained modeling itself has been widely used in many engineering fields~\cite{mesbah2014stochastic,vitus2015stochastic,wang2017chance}.
The moment bounding method offers a provably sufficient condition of the chance constraint. However, the bounding gap may be too large in many applications~\cite{van2016generalized}.
The resulting overly-reduced feasible region may lead to an over-conservative design. 

\textbf{Paper contributions.}
This paper proposes a novel \textbf{Po}lynomial \textbf{B}ounding method for chance-constrained yield-aware \textbf{O}ptimization (PoBO) under truncated non-Gaussian correlated variations. Leveraging the recent uncertainty quantification techniques~\cite{cui2018stochastic,cui2020chance}, PoBO employs and modifies the idea of kinship functions~\cite{feng2010kinship} to approximate the original chance constraints with a better polynomial bounding method. PoBO provides a less conservative design than moment-based bounding methods~\cite{cui2020chance} while ensuring a pre-specified yield requirement. 
The specific contributions of this paper include:
\begin{itemize}[leftmargin=*]
    \item A better bounding method of the chance constraints via optimal polynomial kinship functions. Compared with existing work in the control community~\cite{feng2010kinship}, we avoid the assumption of the independence among random variables and the convexity of kinship functions. The relaxation allows more general non-Gaussian correlated uncertainty modeling and tightens the bounding functions.
    Within a family of polynomial functions, the optimal polynomial kinship functions can be efficiently solved via semidefinite programming.
    Our bounding method preserves the polynomial formulation of the provided surrogate models. It enables the advanced polynomial optimization solvers, which provide a sequence of convex relaxations via semidefinite optimization and searches for the global design. 
    \item Numerical implementation of the PoBO framework. Based on available uncertainty quantification solvers, we implement PoBO efficiently based on some pre-calculated optimal polynomial kinship functions and quadrature samples and weights without requiring any additional circuit simulations.
    \item Validations on a synthetic function and two photonic IC design examples. Our method offers better design performance while meeting the pre-specified yield requirements. This method requires a small number of circuit simulations due to its compatibility with recent data-efficient uncertainty quantification methods~\cite{zhang2016big,cui2018stochastic}.
\end{itemize}

While this work focuses on the fundamental theory, algorithms, and their validation on small-size photonic circuits, the proposed method can be combined with sparse or low-rank surrogate modeling methods~\cite{li2010finding,zhang2016big} to handle large-scale design cases with much more design variables and process variations.


\section{Background}
\label{sec:preliminaries}
This section reviews  chance-constrained yield-aware optimization and its implementation via moment bounding~\cite{cui2020chance}.
\subsection{Chance-Constrained Yield-Aware Optimization}
We denote the design variables by $\mat{x} = [x_1, x_2, \ldots, x_{d_1}]^T \in \mat{X}$, and the process variations by random parameters $\vecpar = [\xi_1, \xi_2, \ldots, \xi_{d_2}]^T \in \mat{\Xi}$.
Let $\{y_i (\mat{x},\vecpar)\}_{i=1}^n$ denote $n$ performance metrics that are considered in yield estimation, $\{u_i\}_{i=1}^n$ denote their corresponding upper bounds specifying the design requirements. An indicator function is defined as 
\begin{equation}\label{eq:indicator}
    I(\mat{x},\vecpar)=
    \begin{cases}
    1, & y_i(\mat{x},\vecpar) \le u_i, \forall i=[n]; \\
    0, & \text{otherwise}.
    \end{cases}
\end{equation}
Here $[n]=\{ 1,2,\cdots, n\}$. The yield at a certain design choice $\mat{x}$ is defined as
\begin{equation}\label{eq:yield}
    Y(\mat{x}) = \P_{\vecpar} (\mat{y}(\mat{x},\vecpar)\le \mat{u}) = \E_{\vecpar} [I(\mat{x},\vecpar)].
\end{equation}
In conventional yield optimization, one often tries to achieve the best possible yield. This often requires losing remarkable design performance $f(\mat{x}, \vecpar)$ in order to achieve a high yield. 

Simply maximizing the yield may lead to an {\bf over-conservative} design. As an example shown in Fig.~\ref{fig:yield_obj}, one may lose lots of performance (from $2.2$ to $1.4$) while just getting marginal yield improvement from $99\%$ to $100\%$.    
\begin{figure}[t]
    \centering
    \includegraphics[width = 3.3in]{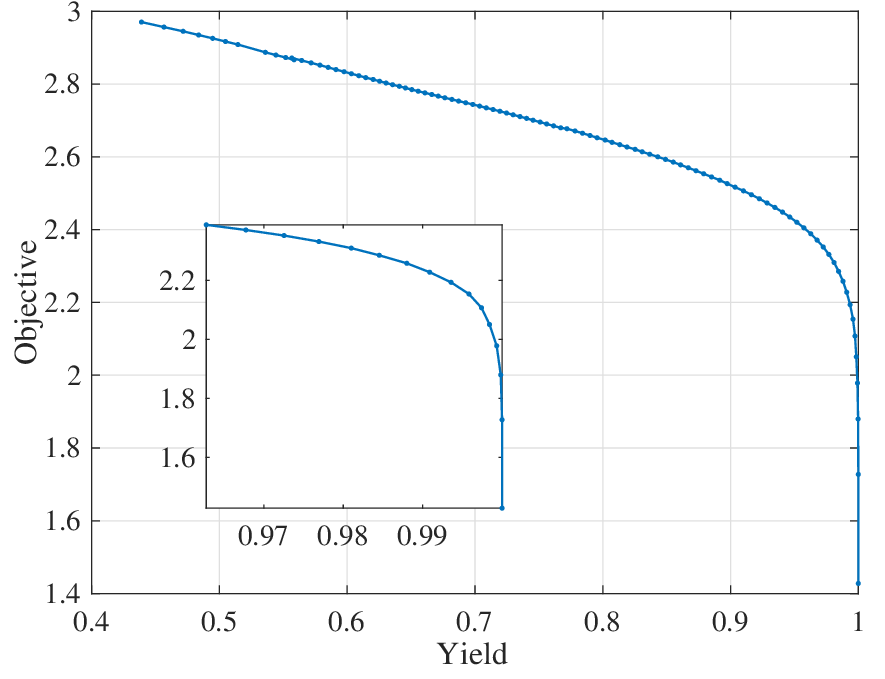}
    \caption{The trade-off between yield and an objective design performance for the example from~\ref{sec:synthetic_example}.
    Simply maximizing the yield can lead to over-conservative performance.}
    \label{fig:yield_obj}
\end{figure}
In order to avoid an over-conservative design, a chance-constrained optimization was proposed in~\cite{cui2020chance}: 
\begin{subequations}\label{eq:CC}
\begin{align}
        \max_{\mat{x} \in \mat{X}} \quad& \E_{\vecpar} [f(\mat{x},\vecpar)]\\
        \st \quad & \P_{\vecpar}(y_i(\mat{x,} \vecpar) \le u_i) \ge 1-\epsilon_i, \forall i=[n]. \label{eq:CC_constraint}
\end{align}
\end{subequations}
where $f(\mat{x}, \vecpar)$ is the performance metric that we aim to optimize, and $\epsilon_i \in [0,1]$ is a risk level to control the probability of meeting each design constraint. Instead of simply maximizing the yield, the chance-constrained optimization tries to achieve a good balance between yield and performance: one can optimize the performance $f(\mat{x}, \vecpar)$ while ensuring a high yield. The circuit yield can be controlled by $\epsilon_i$: reducing $\epsilon_i$ ensures a lower failure rate and thus a higher yield. 

The chance-constrained optimization~\eqref{eq:CC} is generally hard to solve. Firstly, the feasible set produced by the chance constraints is often non-convex and hard to estimate. Secondly, it is also expensive to estimate the design objective function $f(\mat{x}, \vecpar)$ and design constraint function $y_i(\mat{x}, \vecpar)$ due to the lack of analytical expressions. Fortunately, a moment bounding method was combined with uncertainty quantification techniques in~\cite{cui2020chance} to make the problem tractable.

\subsection{Moment Bounding Method for~\eqref{eq:CC}}
Employing the Chebyshev-Cantelli inequality, one can ensure the chance constraint via a moment bounding technique~\cite{calafiore2006distributionally,cui2020chance}. Specifically, with the first and second-order statistical moments of the constraint function, one can convert the probabilistic constraint in~\eqref{eq:CC} to a deterministic one:
\begin{equation}\label{eq:CCineq}
    \begin{aligned}
        \max_{\mat{x} \in \mat{X}} \quad& \E_{\vecpar} [{f}(\mat{x},\vecpar)]\\
        \st \quad & \E_{\vecpar}[{y}_i(\mat{x}, \vecpar)] + \gamma_{\epsilon_i}\sqrt{\text{Var}_{\vecpar} [{y}_i(\mat{x},\vecpar)]} \le u_i, \forall i \in [n],
    \end{aligned}
\end{equation}
where constant $\gamma_{\epsilon_i} = \sqrt{\frac{1-\epsilon_i}{\epsilon_i}}$. When the objective and constraint functions are described by certain surrogate models such as generalized polynomial chaos~\cite{zhang2013stochastic,cui2018stochastic}, one can easily extract their mean and variances. This can greatly simplify the problem and reduce the computational cost, as shown by the yield-aware optimization of photonic IC in~\cite{cui2020chance}.

The moment constraint in~\eqref{eq:CCineq} is a sufficient but unnecessary condition of the original chance constraint in~\eqref{eq:CC}. Therefore, any feasible solution of~\eqref{eq:CCineq} should satisfy the probability constraint of~\eqref{eq:CC}.   However, the feasible set produced by a moment bounding can be much smaller than the exact one~\cite{van2016generalized}. This usually leads to an over-conservative design solution. When the risk level is very small, the moment bounding method may even produce an empty feasible set, leading to an unsolvable problem (see Section~\ref{sec:numerical_results}).

\begin{figure}[t]
    \centering
    \includegraphics[width = 3.3in]{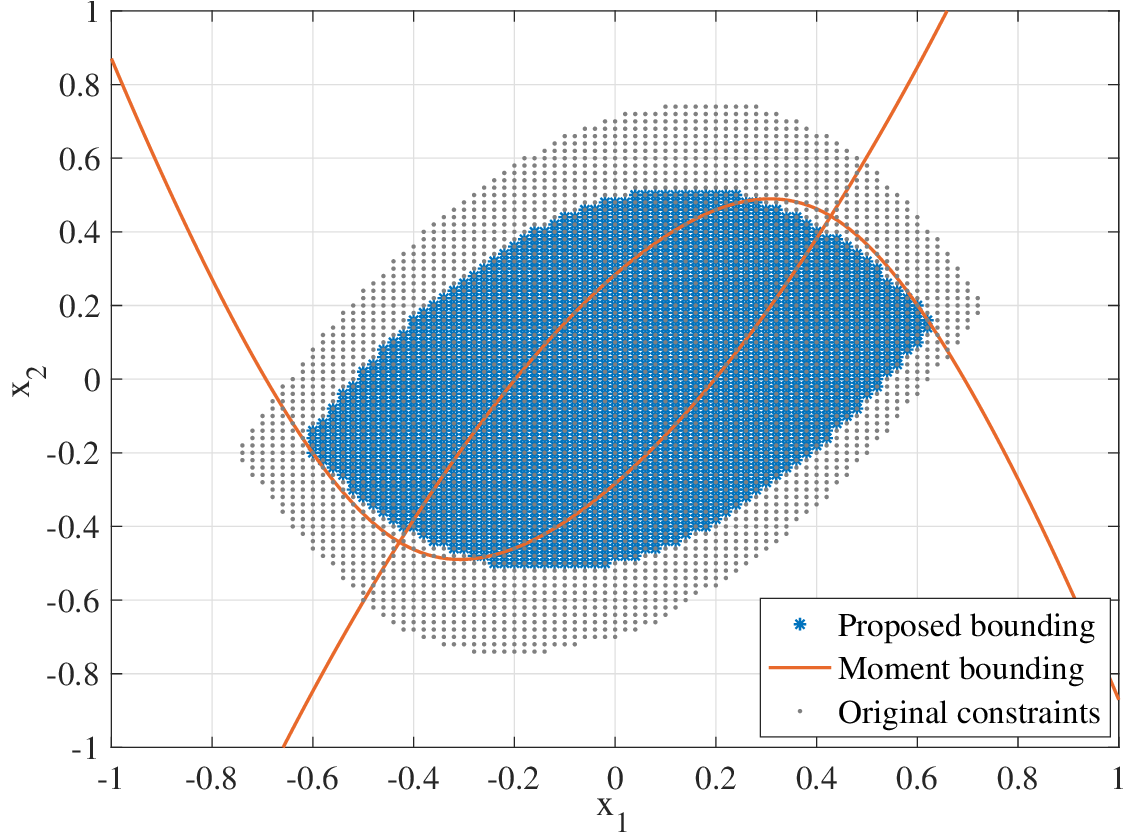}
    \caption{The feasible set of an original chance constraint, moment bounding (the cross region of two orange lines), and the proposed polynomial bounding.} 
    \label{fig:single_feasible}
\end{figure}

\section{Proposed Polynomial Bounding Method}
\label{sec:proposed_method}
To avoid the possible over-conservative bounding of the moment methods~\cite{cui2020chance}, we propose to bound the chance constraint via a more accurate polynomial method. Fig.~\ref{fig:single_feasible} plots the feasible regions obtained by different bounding methods for the synthetic example in Sec.~\ref{sec:synthetic_example}. 
For the same chance constraints, the moment bounding method produces a feasible set that is much smaller than the exact one, whereas our polynomial bounding method generates a better approximation of the feasible set. Due to the more accurate approximation of the feasible set, our proposed polynomial bounding method can provide a less conservative design in yield-aware optimization. Now we describe how to generate the polynomial bounds via kinship functions. 

\subsection{Kinship Function}
The kinship function was first proposed to construct a convex approximation of an indicator function in~\cite{feng2010kinship}.
We generalize the concepts of~\cite{feng2010kinship} with two relaxations: 
\begin{itemize}
    \item We do not require the convexity of a kinship function.
    \item We do not require the random variables $\vecpar$ to be mutually independent. Instead, we consider the more challenging cases where random parameters are truncated non-Gaussian correlated. 
\end{itemize}

We slightly modify the definition of a kinship function.

\begin{definition}[\textbf{Kinship function}]\label{def:kinship}
A kinship function $\kappa (z): [-1, \infty) \rightarrow \R$ is a function that satisfies the following constraints:
\begin{itemize}
    \item $\kappa(z) = 1$ when $z=0$;
    \item $\kappa(z)\geq 0$ for any $z \in [-1, +\infty) $;
    \item $\kappa(z_1) \geq \kappa (z_2)$ for any $z_1 \geq z_2$ in the range $[-1,\infty)$.
\end{itemize} 
\end{definition}

Based on kinship functions, the following theorem offers an upper bound for any probability of constraint violations.
\begin{theorem}[\textbf{Risk integral}~\cite{feng2010kinship}]\label{theorem:property}
Let $\kappa(\cdot)$ be a kinship function, $g(\mat{x},\vecpar)\geq -1$, $\mu(\vecpar)$ be the density function of random vector $\vecpar$, and $V_{\kappa}(\mat{x})$ be a risk integral quantity:
\begin{equation}\label{eq:risk_integral}
    V_{\kappa}(\mat{x}) \coloneqq \int \limits_{\mat{\Xi}} \kappa[g(\mat{x},\vecpar)]\mu(\vecpar) d\vecpar,
\end{equation}
then we have $\P\{\vecpar \in \mat{\Xi}: g(\mat{x},\vecpar)>0\}\le V_{\kappa}(\mat{x})$.
\end{theorem}
\begin{proof}
According to Definition~\ref{def:kinship}, $\kappa[g(\mat{x},\vecpar)]$ is nonnegative in $[-1,\infty)$ and greater than 1 if $g(\mat{x},\vecpar) \ge 0$. Therefore, for any probability measure $\mu(\vecpar)$ on $\mat{\Xi}$, we have:
\begin{subequations}
\small
\begin{align}
    V_{\kappa}(\mat{x}) &\ge \int \limits_{\{\vecpar \in \mat{\Xi}, g(\mat{x},\vecpar)>0\}} \kappa[g(\mat{x},\vecpar)]\mu(\vecpar) d\vecpar \label{eq:risk_integral_ineq} \\
    & \ge \int \limits_{\{\vecpar \in \mat{\Xi}: g(\mat{x},\vecpar)>0\}} \mu(\vecpar) d\vecpar = \P\{\vecpar \in \mat{\Xi}, g(\mat{x},\vecpar)>0\}.
\end{align} \normalsize
\end{subequations}
\end{proof}
There exist many possible choices of kinship functions. Next, we will show how to pick some polynomial kinship functions. We consider the polynomial function family because it is compatible with existing surrogate modeling techniques~\cite{zhang2013stochastic,zhang2016big,li2010finding} to facilitate yield-aware optimization.
\begin{figure}[t]
    \centering
    \includegraphics[width = 3.3in]{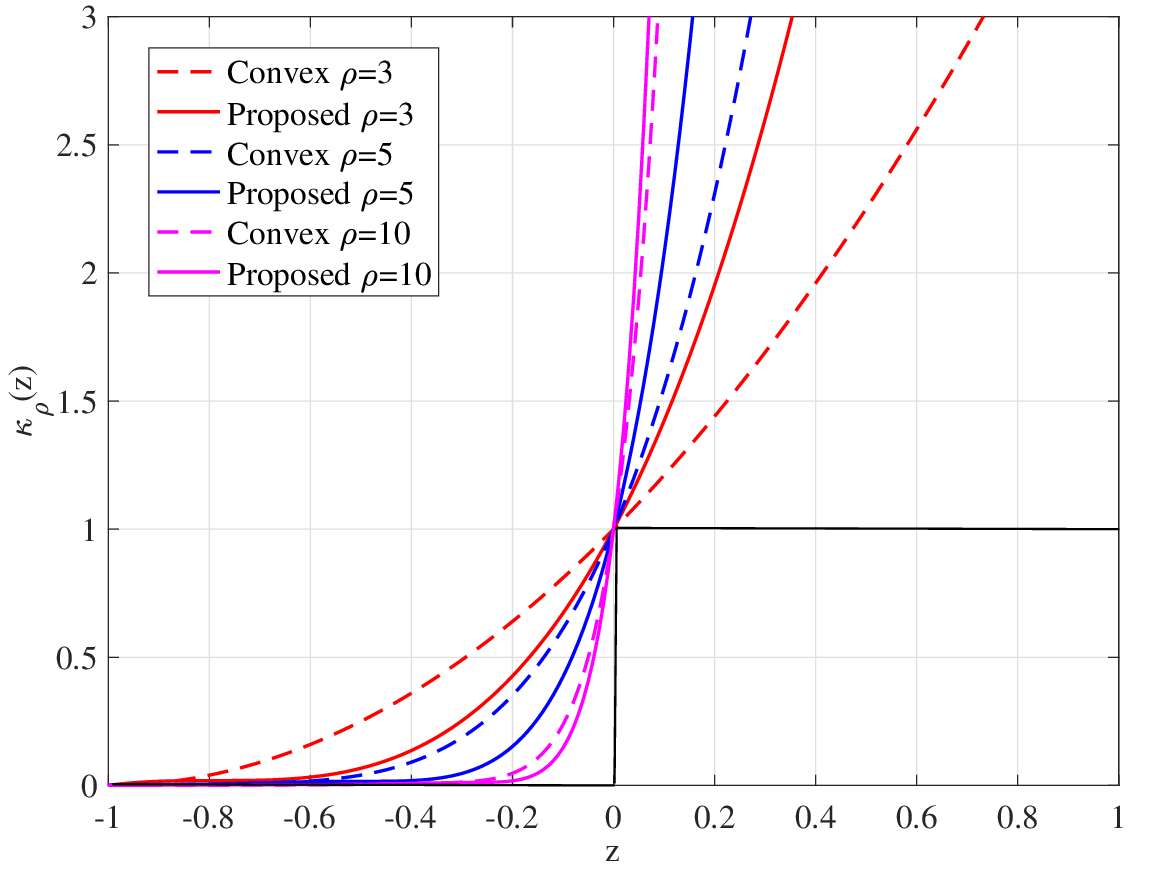}
    \caption{Different kinship functions. Black curve: an indicator function.}
    \label{fig:kinship}
\end{figure}
\subsection{Optimal Polynomial Kinship Function}
An optimal kinship function is defined as a kinship function that minimizes its integral over $[-1,0]$~\cite{feng2010kinship}:
\begin{equation}
    \kappa^\star(\cdot) \coloneqq \argmin \limits_{\kappa(\cdot)\in {\cal K}} \int_{-1}^0 \kappa(z)dz.
\end{equation}
Here ${\cal K}$ is the set of all possible kinship functions that satisfy the constraints in Definition~\ref{def:kinship}. Let $z=g(\mat{x},\vecpar)$, then the above definition can be understood as minimizing the gap between the left- and right-hand sides of~\eqref{eq:risk_integral_ineq}. 

Now we consider choosing a kinship function from a family of order-$\rho$ polynomials ${\cal K}_{\rho} \coloneqq \{\kappa (z) | \kappa(z) = \sum_{i=0}^\rho \zeta_{i} z^{i} \}$. The optimal polynomial kinship function, denoted as $\kappa_{\rho}(\cdot)$, can be constructed by solving the following optimization problem:
\begin{equation}\label{eq:poly_kinship}
\begin{array}{cc}
 \min \limits_{\zeta_0,\ldots,\zeta_\rho} \quad & \int \limits_{-1}^0 \kappa_{\rho}(z)dz \\
    \st \quad & \kappa_{\rho}(z) = \zeta_0 + \zeta_1 z + \cdots + \zeta_\rho z^\rho \in {\cal K}_\rho, \\
    & \kappa_{\rho}(0) = 1, \; \kappa_{\rho}(-1) = 0,\\
    & \kappa_{\rho}^\prime(z) \ge 0, \; \forall z \ge -1.
\end{array}
\end{equation}
The optimization problem can be reformulated as semidefinite programming and the details are provided in Appendix~\ref{sec:optimal_poly}. The obtained optimal polynomial kinship function $\kappa_{\rho}(\cdot)$ with different orders are visualized in Fig.~\ref{fig:kinship}. Since we relax the convexity requirement, given the same polynomial order, the proposed polynomial function is a tighter approximation to the indicator function than a convex one~\cite{feng2010kinship} for $z \in [-1, 0]$.

Based on the obtained optimal polynomial kinship function $\kappa_{\rho}(\cdot)$, we bound the original chance constraint~\eqref{eq:CC_constraint} by enforcing the upper bound of the failure probability below $\epsilon_i$:
\begin{equation}
    \P_{\vecpar}(y_i(\mat{x,} \vecpar) > u_i) \leq\\
    \int \limits_{\mat{\Xi}} \kappa_{\rho}(y_i(\mat{x},\vecpar)-u_i) \mu(\vecpar) d\vecpar \le \epsilon_i.
\end{equation}
The order of $\kappa_{\rho}(\cdot)$ controls the upper bound of violation probability.
When $\rho$ is small, the bounding gap in~\eqref{eq:risk_integral_ineq} is large.
As $\rho \rightarrow \infty$, the polynomial bounding leads to a worst-case robust design optimization according to Theorem 3 in~\cite{feng2010kinship}. 
This is equivalent to setting a risk level $\epsilon_i = 0$, which leads to an extremely over-conservative design. 
Fortunately, this is not a trouble in practice since we do not use a very high-order polynomial due to the computational issues.
In practice, there exists an optimal order $\rho^\star$ for bounding the violation probability most accurately. The optimal $\rho^\star$ is unknown {\it a-priori}, but heuristically we find that setting $\rho \in [5,10]$ usually offers an excellent bound.

The proposed bounding method can be extended to deal with joint chance constraints by constructing a multivariate polynomial kinship function. We plan to report the results in a future paper.


\section{The PoBO Framework}
\label{sec:implementation}
Based on the proposed polynomial bounding for chance constraints, we further present the novel PoBO method to achieve less conservative yield-aware optimization. 

\subsection{Workflow of PoBO}
Our PoBO framework has two weak assumptions on the design and random variables: 
\begin{itemize}
    \item The design variable $\mat{x}$ is box-bounded, \ie, ~$\mat{x} \in \mat{X} = [a,b]^{d_1}$. This is normally the case in circuit optimization.
    \item The process variations $\vecpar$ are truncated and non-Gaussian correlated with a joint probability density function $\mu(\vecpar)$. 
\end{itemize}
The 2nd assumption is not strong at all. Many practical process variations are correlated and not guaranteed to be Gaussian. Additionally, the values of almost all practical geometric or material parameters are bounded, although some simplified unbounded distributions (e.g., Gaussian distributions) were used in previous literature for ease of implementation.

The overall flow of PoBO is summarized below.
\begin{itemize}[leftmargin=*]
\item \textbf{Step 1:} Surrogate modeling. 
We use the recent uncertainty quantification solver~\cite{cui2018stochastic} to construct polynomial surrogate models for the objective and constraint functions, \ie, $f (\mat{x},\vecpar) \approx \hat{f}(\mat{x},\vecpar)$ \text{and}     $y_i(\mat{x},\vecpar) \approx \hat{y}_i(\mat{x},\vecpar), \forall i=[n]$.

\item \textbf{Step 2:} Bounding the chance constraints via the proposed optimal polynomial kinship functions. 
This transforms a chance-constrained probabilistic optimization problem into a tractable deterministic one with a high-quality solution. 

\item \textbf{Step 3:} Design optimization. 
We use a polynomial optimization solver, e.g., the semidefinite programming relaxation~\cite{henrion2009gloptipoly}, to obtain a globally optimal solution. 
\end{itemize}

The PoBO framework reformulates the original chance-constrained optimization~\eqref{eq:CC} to the following optimization: 
\begin{subequations}\label{eq:CC_poly_kinship}
\small
\begin{align}
        \max_{\mat{x} \in \mat{X}} \quad& \E_{\vecpar} [\hat{f}(\mat{x},\vecpar)] \\
        \st \quad & V_{\kappa}^{(i)}(\mat{x}) =  \int \limits_{\mat{\Xi}} \kappa_{\rho}(\upsilon_i(\mat{x},\vecpar)) \mu(\vecpar) d\vecpar \le \epsilon_i, \; \forall i \in [n]. 
        \label{eq:integral_in_cc}
\end{align} \normalsize
\end{subequations}
Here $\upsilon_i(\mat{x},\vecpar) \coloneqq \hat{y}_i(\mat{x},\vecpar)-u_i$, $\hat{f} (\mat{x}, \vecpar)$ and $\hat{y}_i (\mat{x}, \vecpar)$ are the polynomial surrogate models of $f(\mat{x}, \vecpar)$ and $y_i (\mat{x}, \vecpar)$, respectively.

In the following subsections, we will describe the implementation details of this PoBO framework.

\subsection{Building Surrogate Models}
\label{sec:build_surrogate}
High-quality performance models are important to speed up design optimization.
We employ the advanced stochastic collocation method with non-Gaussian correlated uncertainty~\cite{cui2018stochastic,cui2020chance} in Step 1. 
This method approximates a smooth stochastic function as the linear combination  of some orthogonal and normalized polynomial basis functions:
\begin{equation}\label{eq:gPC_approx}
f(\mat{x},\vecpar) \approx \hat{f}(\mat{x},\vecpar) = 
\sum_{\lvert \basisInd \rvert + \lvert \boldsymbol{\beta} \rvert = 0}^p c_{\basisInd,\boldsymbol{\beta}} \mat{\Phi}_{\basisInd} (\mat{x}) \mat{\Psi}_{\boldsymbol{\beta}} (\vecpar). 
\end{equation}
Here $\basisInd$ and $\boldsymbol{\beta}$ are two index vectors, $\mat{\Phi}_{\basisInd}(\mat{x})$ and $\mat{\Psi}_{\boldsymbol{\beta}}(\vecpar)$ are two series of orthogonal polynomial basis functions, and $p$ upper bounds the total order of the product of two basis functions.
The corresponding coefficients $c_{\basisInd,\boldsymbol{\beta}}$ are calculated via a projection method using some optimization-based quadrature samples and weights of $\mat{x}$ and $\vecpar$~\cite{cui2018stochastic}. When the parameter dimensionality is not high, this method only needs a small number of simulation samples to produce a highly accurate surrogate model with a provable error bound. 
When the number of dimensions becomes high, we could utilize many existing advanced uncertainty quantification techniques to model the performance more efficiently~\cite{he2021high,cui2019high}. 
It is also possible to extend the proposed chance-constrained yield-aware method to other types of performance models.
\subsection{Scaling the Yield Metrics ${\upsilon}_i(\mat{x},\vecpar)$} 
\label{sec:lower_bound}
To bound the failure probability $\P\{\vecpar \in \mat{\Xi}: \upsilon_i(\mat{x},\vecpar)>0\}$ via the optimal kinship function, $\upsilon_i(\mat{x},\vecpar)$ must be in the range $[-1,\infty)$ according to Definition~\ref{def:kinship}. Once $\upsilon_i(\mat{x},\vecpar)$ is lower bounded, we can always scale it to meet this requirement. 

Since $\upsilon_i(\mat{x},\vecpar)$ is a polynomial function in our problem setting, and both $\mat{x}$ and $\vecpar$ are assumed bounded, we can compute the minimum value ${\upsilon_i^-} \coloneqq {\min \limits_{\mat{x} \in \mat{X}, \vecpar \in \mat{\Xi}} \upsilon_i(\mat{x},\vecpar) }$. 
Then we change the lower bound of $\upsilon_i(\mat{x},\vecpar)$ to -1 as follows:
\begin{equation}\label{eq:lower_bound}
\begin{matrix}
\upsilon_i(\mat{x},\vecpar) \longleftarrow  -\frac{1}{\upsilon_i^-} \upsilon_i(\mat{x},\vecpar), \; \forall i \in [n]. \end{matrix}
\end{equation}
The scaling factor $-\frac{1}{\upsilon_i^-}$ is positive as long as the problem \eqref {eq:CC_poly_kinship} is solvable. This is because the existence of $\mat{x} \in \mat{X}$ and $ \vecpar \in \mat{\Xi}$ such that  $\upsilon_i(\mat{x},\vecpar) < 0$ is the necessary condition to satisfy the yield constraint. We can easily avoid $\upsilon_i^- = 0$ by adding a sufficiently small perturbation.
\subsection{Calculating Risk Integral $V_{\kappa}^{(i)}(\mat{x})$}
\begin{figure}[t]
    \centering
    \includegraphics[width = 3.3in]{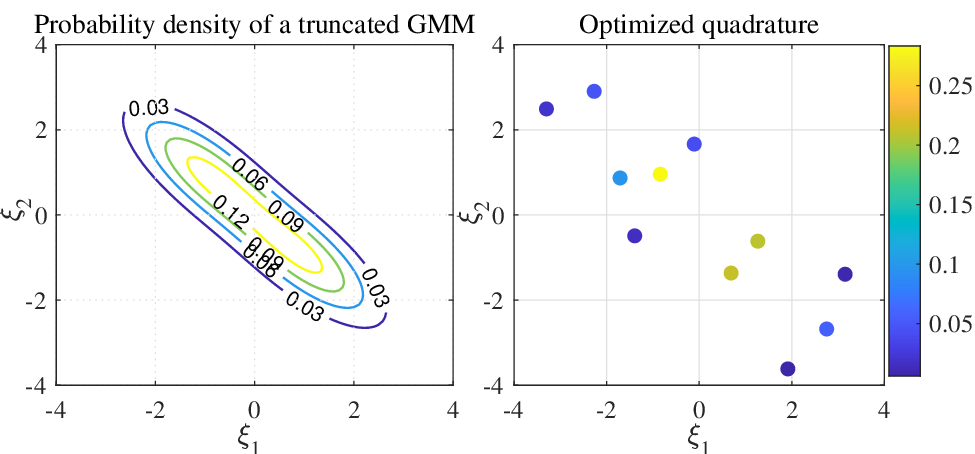}
    \caption{An example of optimized quadrature rule for a two-dimensional truncated Gaussian mixture model (GMM) with two components. 
    The color bar of the right figure represents the weights of all samples.
    The shown quadrature rule satisfies the exact integration up to order 6.}
    \label{fig:quadrature}
\end{figure}
In order to upper bound the probability of violating a design constraint, we need to calculate the integration in~\eqref{eq:integral_in_cc}. 
This can be a challenging task for a truncated non-Gaussian correlated random vector $\vecpar$ since classical numerical quadrature rules~\cite{golub1969calculation,gerstner1998numerical} do not work for non-Gaussian correlated variables.
Fortunately, we can reuse the quadrature rule of $\vecpar$ as the by-product of building the surrogate models in Sec.~\ref{sec:build_surrogate}. 
Specifically, the quadrature points and weights $\{\vecpar_l, w_l\}_{l=1}^{M}$ compute the exact integration up to order $2q$, obtained by solving the following optimization problem~\cite{cui2018stochastic}:
\begin{equation}\label{eq:quadrature_xi}
\min_{\vecpar_l,w_l \ge 0} \; \sum_{\lvert\boldsymbol{\beta}\rvert=0}^{2q} {\left(\E_{\vecpar}\left[\mat{\Psi}_{\boldsymbol{\beta}}(\vecpar)\right] - \sum_{l=1}^{M} \mat{\Psi}_{\boldsymbol{\beta}}(\vecpar_l) w_l \right)}^2.
\end{equation}
An example of the solved quadrature is shown in Fig.~\ref{fig:quadrature}. The number of quadrature samples could be controlled by tuning the optimization precision.
The detailed accuracy analysis and the bound of $M$ are provided in~\cite{cui2018stochastic}, which is omitted here. 
\begin{algorithm}[t]
  \caption{Flow of the proposed PoBO.} 
  \label{alg:flow}
  \begin{algorithmic}[1]
    \Require Box-bounded design variable $\mat{x} \in \mat{X}$, truncated non-Gaussian correlated variations $\vecpar \in \mat{\Xi}$, risk levels $\mat{\epsilon}$
    \Ensure
      Optimized design $\mat{x}^\star$
    \State Formulate the chance-constrained problem~\eqref{eq:CC}.
    \State Obtain surrogate models for the design objective function $f(\mat{x},\vecpar) \approx \hat{f}(\mat{x},\vecpar)  $ and constraint functions $ y_i(\mat{x},\vecpar)  \approx \hat{y}_i(\mat{x},\vecpar), \forall i \in [n]$.
    \State Scale the yield metrics ${\upsilon_i}(\mat{x},\vecpar)$ via Eq.~\eqref{eq:lower_bound}.  
    \State Compute the optimal polynomial kinship function $\kappa_{\rho}(\cdot)$ via ~\eqref{eq:poly_kinship}.
    \State Calculate risk integral $V_{\kappa}^{(i)}(\mat{x})$ via the quadrature rule obtained from~\eqref{eq:quadrature_xi}. 
    \State Seek the optimal design of problem~\eqref{eq:CC_int_poly_kinship} via a global polynomial optimization solver.
  \end{algorithmic}
\end{algorithm}

Theoretically, we need a quadrature rule to exactly calculate the integration up to order $p \rho$ in~\eqref{eq:integral_in_cc}. 
The exact quadrature rule can be computed offline via solving~\eqref{eq:quadrature_xi} with $q=\lceil \frac{\rho p}{2} \rceil$. 
In practice, a low-order quadrature rule, like $q=p$, often offers sufficient numerical accuracy. 
Therefore, we can directly use the quadrature rule used when building surrogate model~\eqref{eq:gPC_approx} to calculate the risk integral $V_{\kappa}^{(i)}(\mat{x})$. 
Since the quadrature points will not be simulated, it does not introduce any additional computational burdens. 
Based on the quadrature rule, problem~\eqref{eq:CC_poly_kinship} can be converted to the following deterministic constrained polynomial optimization~\eqref{eq:CC_int_poly_kinship}:
\begin{equation}\label{eq:CC_int_poly_kinship}
    \begin{aligned}
        \max_{\mat{x} \in \mat{X}} \quad& \E_{\vecpar} [\hat{f}(\mat{x},\vecpar)] \\
        \st \quad &  
        \sum_{l=1}^{M} w_l \kappa_{\rho}(\upsilon_i(\mat{x},\vecpar_l)) \le \epsilon_i, \quad \forall i \in [n].\\
    \end{aligned}
\end{equation}
Note that the expectation value in the objective function can be easily obtained since $\hat{f}(\mat{x},\vecpar)$ is a generalized polynomial-chaos expansion~\cite{zhang2013stochastic,cui2018stochastic}.

\subsection{Algorithm Summary}

We summarize PoBO in Alg.~\ref{alg:flow}. Below are some remarks: 
\begin{itemize}
\item In line 4 of Alg.~\ref{alg:flow}, the optimal polynomial kinship functions~\eqref{eq:poly_kinship} can be computed offline and stored as a look-up table. 
\item In line 5 of Alg.~\ref{alg:flow}, the optimization-based quadrature rule in~\cite{cui2018stochastic} can be used to calculate the risk integral without any additional simulations. The quadrature rule can be computed offline as well.
\item This method enables a global polynomial optimization solver to obtain the optimal design of \eqref{eq:CC_int_poly_kinship}.
\end{itemize}
The curse of dimensionality could be a challenge for Line 2 and Line 6.
For the surrogate modeling step, we can utilize some high-dimensional uncertainty quantification techniques~\cite{he2021high,cui2019high} to reduce the cost.
In the design optimization step, the current limitation comes from the polynomial optimization solver. 
Typically, the polynomial optimization problem can be reformulated as a convex moment problem. Under very mild assumptions, we can build a series of semidefinite programming problems whose solutions are proved to converge monotonically and asymptotically to the global optimum~\cite{lasserre2001global,lasserre2008semidefinite, nie2014optimality}.
The relaxed semidefinite programming problems have the size of $O({d_1}^p)$ with the number of design variables ${d_1}$ and polynomial order $p$. Although it grows polynomially with the number of design variables, it can be challenging when $p$ is high.
Fortunately, the design optimization does not suffer from the number of dimensions of process variations ${d_2}$.
The challenge caused by the high dimensionality of $\mat{x}$ may be addressed in the future by using other nonlinear optimization solvers, or a better polynomial optimization solver (e.g. a sparse polynomial solver) that can exploit the sparse structure of the polynomial surrogate.

\section{Numerical Results}
\label{sec:numerical_results}
In this section, we validate the proposed PoBO framework via the synthetic example and two realistic photonic IC examples from~\cite{cui2020chance}.
The polynomial optimization is solved via GloptiPoly 3~\cite{henrion2009gloptipoly}, which is a global optimization solver based on hierarchical semidefinite programming. Our codes are implemented in MATLAB and run on a computer with a 2.3 GHz CPU and 16 GB memory. 
\begin{table}[t]
\centering
\caption{Optimization Results for Synthetic Function}
\label{tab:syn_result}
\begin{adjustbox}{width=3.3in}
\begin{threeparttable}
\begin{tabular}{cccccc}
\toprule
Risk level $\epsilon$  &Method & Objective &  $\Delta_1$ (\%) & $\Delta_2$ (\%) & Yield (\%)\\
\midrule 
\multirow{2}{*}{0.01} 
&Moment~\cite{cui2020chance} &N/A\tnote{*}     &N/A\tnote{*}   &N/A\tnote{*}  &N/A\tnote{*}\\
&\textbf{Proposed} &\textbf{1.14}   &\textbf{1.01}  &\textbf{0.99}  & 99.98   \\
\hline
\multirow{2}{*}{0.05} 
&Moment~\cite{cui2020chance} &1.88     &5.25   &5.26  &99.98  \\
&\textbf{Proposed} &\textbf{2.11}   &\textbf{5.21}  &\textbf{3.97}  & 98.76 \\
\hline
\multirow{2}{*}{0.1}
&Moment~\cite{cui2020chance} &2.19     &\textbf{10.47}   &10.98 &99.36  \\
&\textbf{Proposed} &\textbf{2.26}   &10.67  &\textbf{7.98} & 97.08  \\
\bottomrule
\end{tabular}
\begin{tablenotes}
   \item[*] The algorithm fails with no feasible solution. 
\end{tablenotes}
\end{threeparttable}
\end{adjustbox}
\end{table}

\begin{figure}[t]
    \centering
    \includegraphics[width = 3.3in]{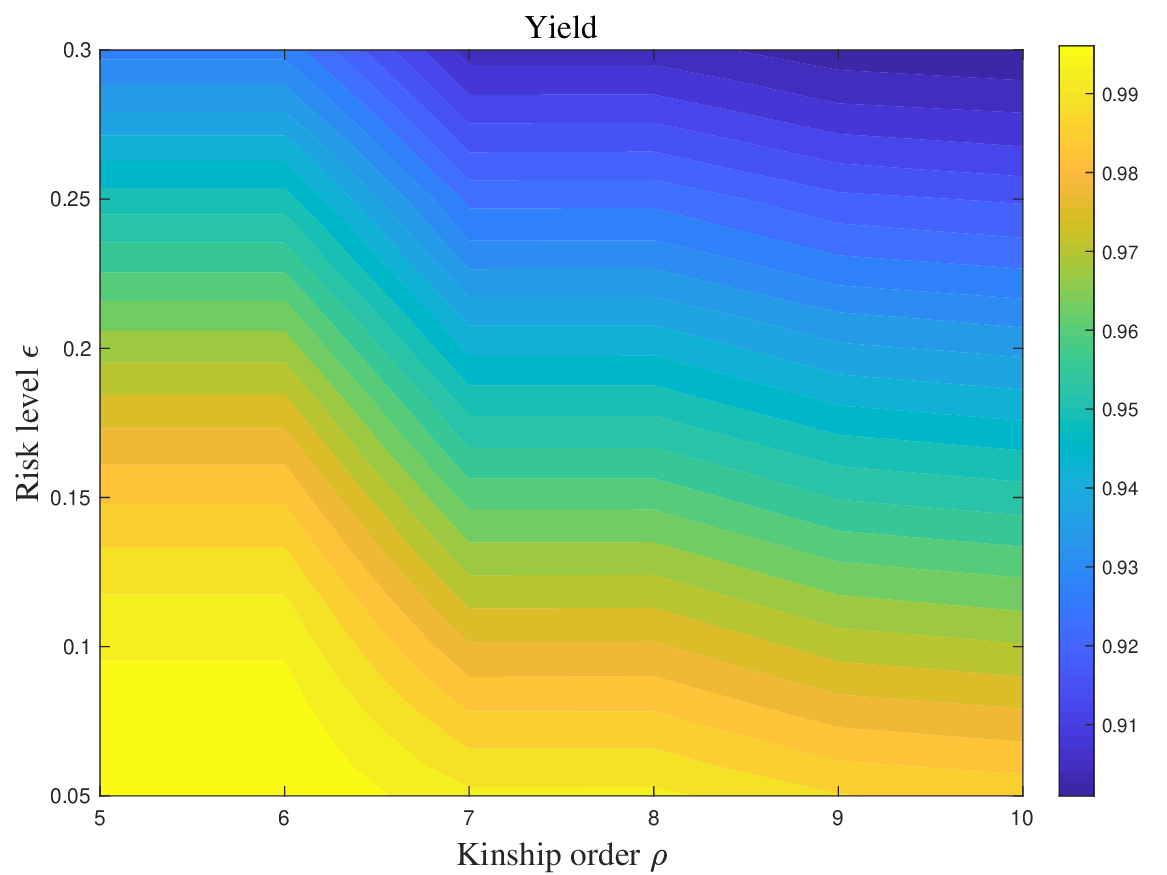}
    \caption{Yield of \eqref{eq:syn_func} given different risk levels $\epsilon$ and polynomial kinship orders $\rho$.}
    \label{fig:different_rho}
\end{figure}

\begin{figure}[t]
    \centering
    \includegraphics[width = 3.3in]{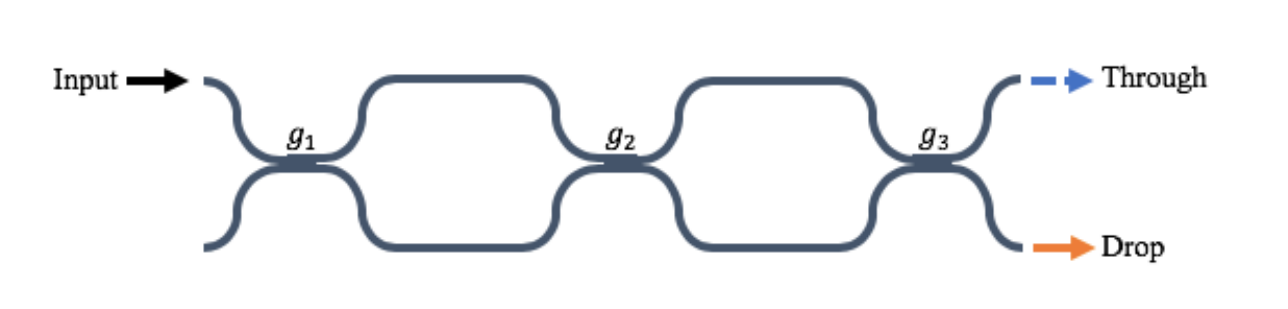}
    \caption{The schematic of a third-order Mach-Zehnder interferometer.}
    \label{fig:benchmark_MZI}
\end{figure}
\begin{figure*}[t!]
    \centering
    \includegraphics[width = 6.6in]{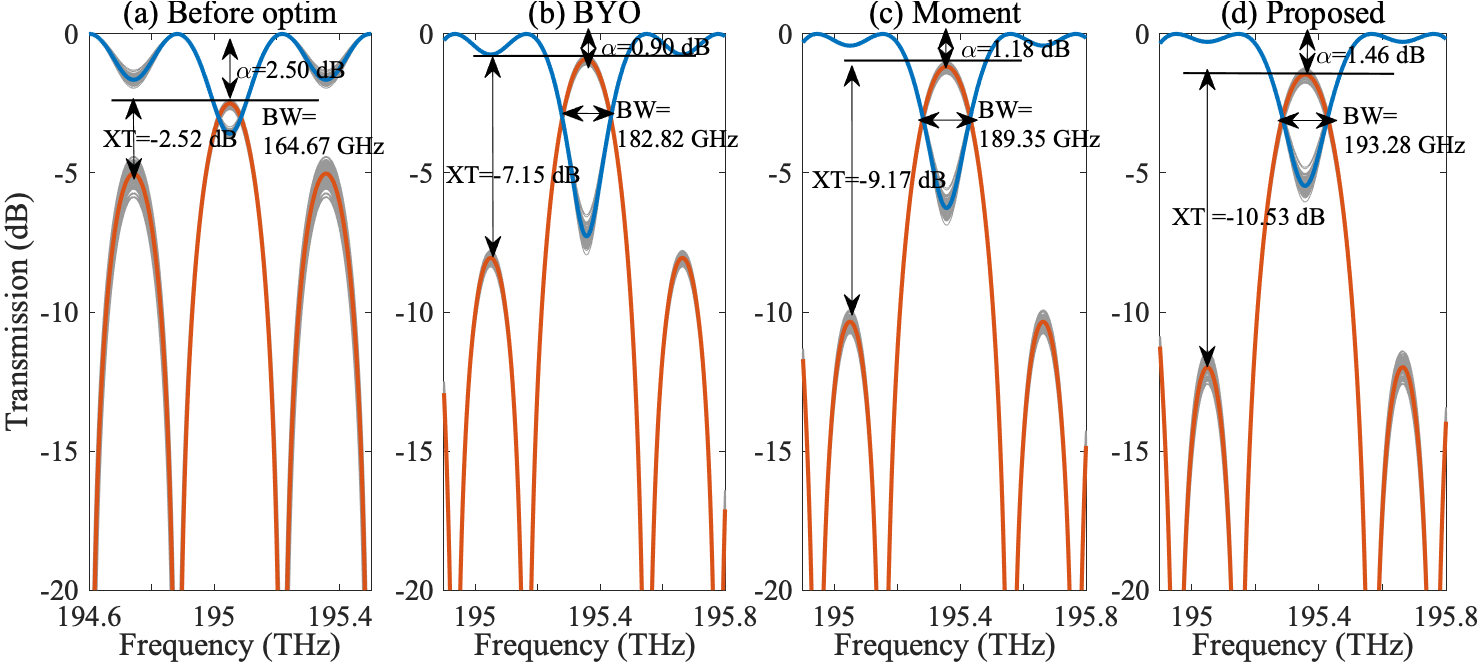}
    \caption{The transmission curves of the MZI. The grey lines show the performance uncertainties. The orange and blue curves show the transmission rates at the drop and through ports, respectively. The mean values of the bandwidth, crosstalk, and attenuation are denoted as BW, XT, and $\alpha$, respectively. (a) The initial design: $\mat{x}$=[150, 150, 150]; (b) Design after Bayesian yield optimization~\cite{wang2017efficient}: $\mat{x}$=[286.63, 170.59, 299.3]; (c) Design with the moment-bounding yield-aware optimization~\cite{cui2020chance}: $\mat{x}$=[300, 149.67, 300]; (d) Design with the proposed PoBO method: $\mat{x}$=[300, 112.15, 300].}
    \label{fig:MZI_opt_result}
\end{figure*}
\begin{table*}[t]
\centering
\caption{Optimization Results for MZI Benchmark}
\label{tab:MZI_result}
\begin{threeparttable}
\begin{tabular}{ccccccc}
\toprule
 Risk level $\epsilon$    & Method  & $\E_{\vecpar}[\text{BW}]$ (GHz) &  $\Delta_1$ (\%) & $\Delta_2$ (\%) & Yield (\%) & Simulation \# \\
\midrule
\multirow{2}{*}{0.05} 
&Moment~\cite{cui2020chance}    &184.53     &5.26  &5.26   & 100 & 35 \\
&\textbf{Proposed}                 &\textbf{190.99}     &\textbf{5.26}   & \textbf{5.26} & 100 & 35 \\
\hline
\multirow{2}{*}{0.07} 
 &Moment~\cite{cui2020chance}     &187.02     &7.53  &7.53   & 100 & 35 \\
&\textbf{Proposed}          &\textbf{192.10}     &\textbf{7.53}   & \textbf{7.2}  & 99.7 & 35 \\
\hline
\multirow{2}{*}{0.1} 
 &Moment~\cite{cui2020chance}    &189.35     &11.11  &11.11 & 100 & 35 \\
 &\textbf{Proposed}           &\textbf{193.28}     &\textbf{11.11}  & \textbf{4.22} & 93.8  & 35 \\
\hline
N/A\tnote{*}   
&BYO~\cite{wang2017efficient}    &182.82   & N/A\tnote{*}  &N/A\tnote{*} &100 & 2020 \\
\bottomrule
\end{tabular}
\begin{tablenotes}
   \item[*] No risk level is defined for BYO method. Correspondingly, no gap $\Delta$ is defined. 
  \end{tablenotes}
\end{threeparttable}
\end{table*}

\begin{figure}[t!]
    \centering
    \includegraphics[width = 3.3in]{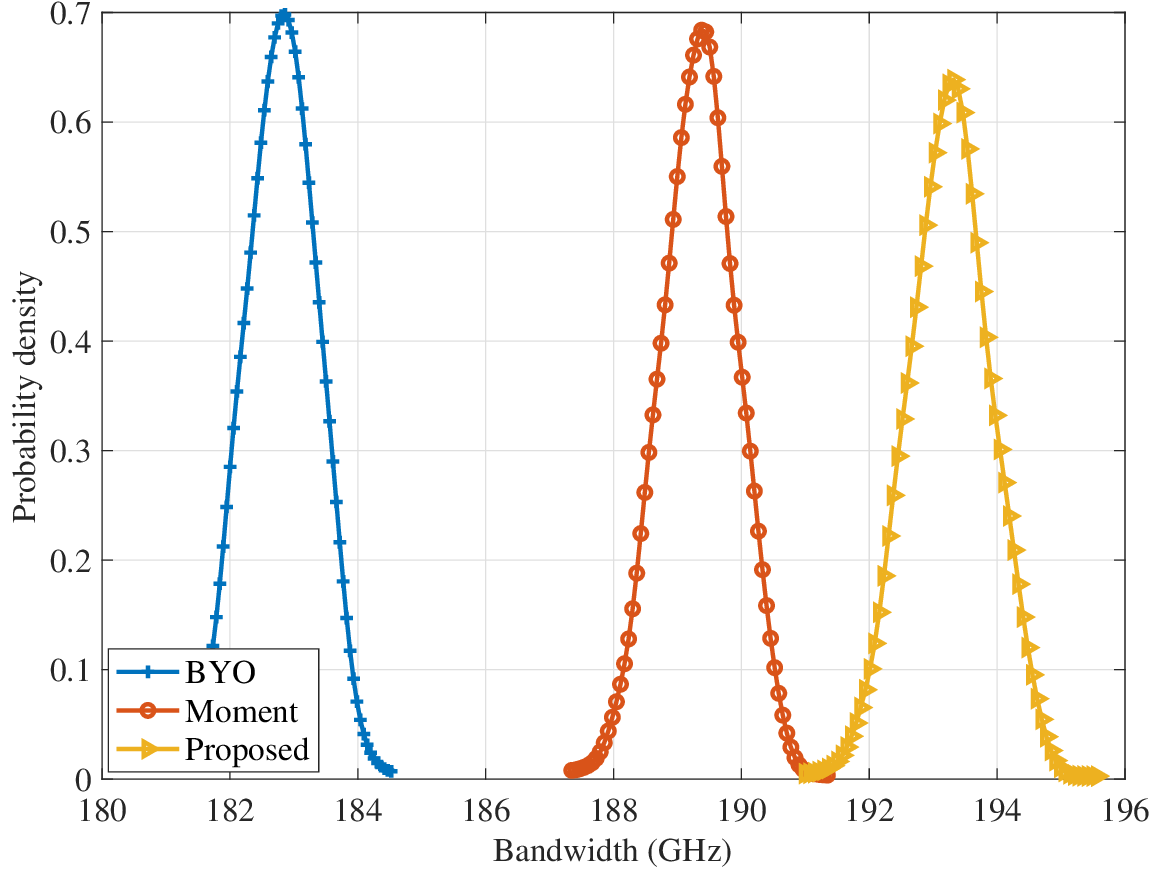}
    \caption{The probability density function of the optimized bandwidth of the MZI by Bayesian yield optimization~\cite{wang2017efficient}, moment bounding~\cite{cui2020chance} and the proposed PoBO (with $\epsilon=0.1$).}
    \label{fig:MZI_BW_pdf}
\end{figure}

\textbf{Baseline Methods.} We choose the moment-bounding chance-constrained optimization~\cite{cui2020chance} as the baseline for comparison. On the photonic IC benchmarks, we further compare our method with the Bayesian yield optimization (BYO) method~\cite{wang2017efficient}, a recent state-of-the-art yield optimization approach. 

\textbf{Gap of chance constraints.}
We modify the indicator function~\eqref{eq:indicator} to define an indicator function ${I_i}(\mat{x},\vecpar)$ for each individual design constraint in yield definition: 
\begin{equation}
{{I_i}(\mat{x},\vecpar)= \begin{cases}
    1, & \hat{y}_i(\mat{x},\vecpar) \le u_i;\\
    0, & \text{otherwise} \end{cases}, \forall i=[n].} \nonumber
\end{equation}
With $N$ random samples, the {\it individual success rate} for each design constraint is evaluated as $ {Y_i}(\mat{x}) = \begin{matrix} \sum_{j=1}^N {I_i}(\mat{x},\vecpar_j)/ N. \end{matrix} $
The gap for the $i$-th chance constraint is the relative difference between $Y_i$ and the pre-specified success rate $1-\epsilon_i$:
\begin{equation}\label{eq:yield_gap}
    \Delta_i = \frac{{Y_i}(\mat{x}) - (1-\epsilon_i)}{1-\epsilon_i}, \; \forall i=[n].
\end{equation}
The chance-constrained optimization can always provide a solution to certify the yield requirement controlled by $\epsilon_i $'s if a feasible solution exists. Therefore, $\Delta_i$ is always non-negative.
Typically, a tighter probabilistic constraint bounding leads to a larger feasible region and allows us to explore the optimal design in a larger space, which is more likely to utilize more risk budgets. Therefore, we use Eq.~\eqref{eq:yield_gap} to measure the gap of feasible regions and the bounding quality.
Notice that we do not attempt to achieve the highest yield.
Instead, our goal is to {\it avoid over-conservative design while ensuring the pre-specified yield requirement}. Therefore, given a certain risk level, we {\bf prefer a smaller gap $\Delta_i$} and a less conservative design solution with better objective performance.

\subsection{Synthetic Function}
\label{sec:synthetic_example}
We first consider a synthetic function with design variables $\mat{x} \in \mat{X}={[-1, 1]}^2$ and random parameters $\vecpar$ following a truncated Gaussian mixture model. Specifically, we assume  
$\mu(\vecpar) = \frac{1}{2}\TN(\bar{\boldsymbol{\mu}}_1,\mat{\Sigma}_1,\mat{a}_1,\mat{b}_1) + \frac{1}{2}\TN(\bar{\boldsymbol{\mu}}_2,\mat{\Sigma}_2,\mat{a}_2,\mat{b}_2)$, with 
$\bar{\boldsymbol{\mu}}_1 = -\bar{\boldsymbol{\mu}}_2= {[0.1, -0.1]}^T$, $\mat{\Sigma}_1 = \mat{\Sigma}_2 =  10^{-2}\begin{bmatrix}
1& -0.75  \\
-0.75 & 1 
\end{bmatrix}$, $\mat{a}_1 = -{[0.2, 0.4]}^T$, $\mat{a}_2 = -{[0.4, 0.2]}^T$,  $\mat{b}_1 = {[0.4, 0.2]}^T$ and $\mat{b}_2 = {[0.2, 0.4]}^T$. Here $\TN(\bar{\boldsymbol{\mu}},\mat{\Sigma},\mat{a},\mat{b})$ denotes a distribution that is a normal distribution with mean $\bar{\boldsymbol{\mu}}$ and variance $\mat{\Sigma}$ in the box $[\mat{a}, \mat{b}]$.

We consider the following chance-constrained optimization:
\begin{equation}\label{eq:syn_func}
    \begin{aligned}
        \max_{\mat{x} \in \mat{X}} \quad & \E_{\vecpar}[3(x_1+\xi_1) - (x_2+\xi_2))]\\
        \st \quad & \P_{\vecpar}({(x_1+\xi_1)}^2+(x_2+\xi_2)\le 1) \ge 1-\epsilon_1,\\
        & \P_{\vecpar}({(x_1+\xi_1)}^2-(x_2+\xi_2)\le 1) \ge 1-\epsilon_2,
    \end{aligned}
\end{equation}
where the two risk levels are set to be equal $\epsilon_1=\epsilon_2=\epsilon$. 
Remark that we do not require $\epsilon_1= \epsilon_2$ since our method naturally handles the individual constraints.

We use 2nd-order polynomials to approximate the three analytical functions and bound the chance constraint with an order-10 optimal polynomial kinship function. 
As shown in Table~\ref{tab:syn_result}, compared with the moment method~\cite{cui2020chance}, the proposed PoBO method produces a better objective value and smaller gaps for chance constraints while meeting the pre-specified yield requirement.  Clearly, a smaller $\epsilon_i$ produces a higher yield. 
The moment bounding method fails to work when $\epsilon=0.01$ while our PoBO can still solve this problem. 
Fig.~\ref{fig:different_rho} shows the obtained yield under different risk levels $\epsilon$ and polynomial kinship orders. The kinship order influences the bounding quality and leads to different yield. However, the results are all of the high quality, leading to certified designs.

\begin{figure}[t]
    \centering
    \includegraphics[width = 3.3in]{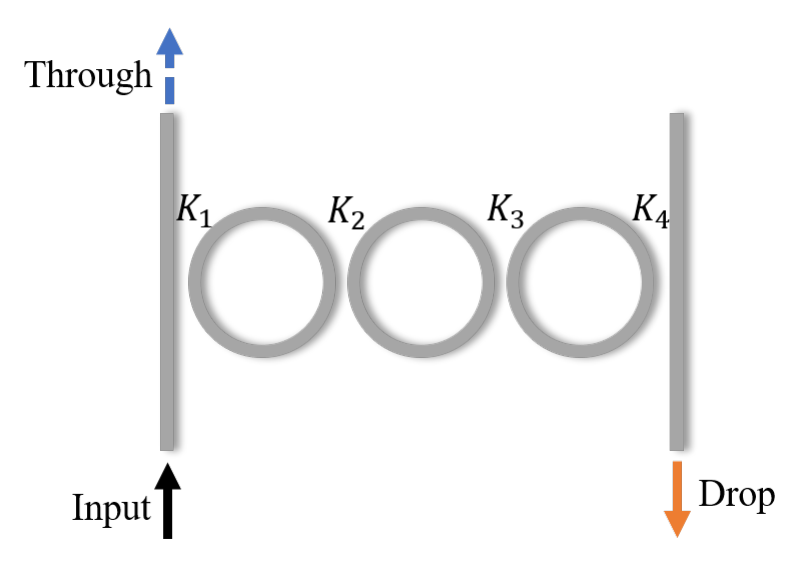}
    \caption{The schematic of a microring add-drop filter.}
    \label{fig:benchmark_Mircoring}
\end{figure}

\begin{figure}[t!]
    \centering
    \includegraphics[width = 3.3in]{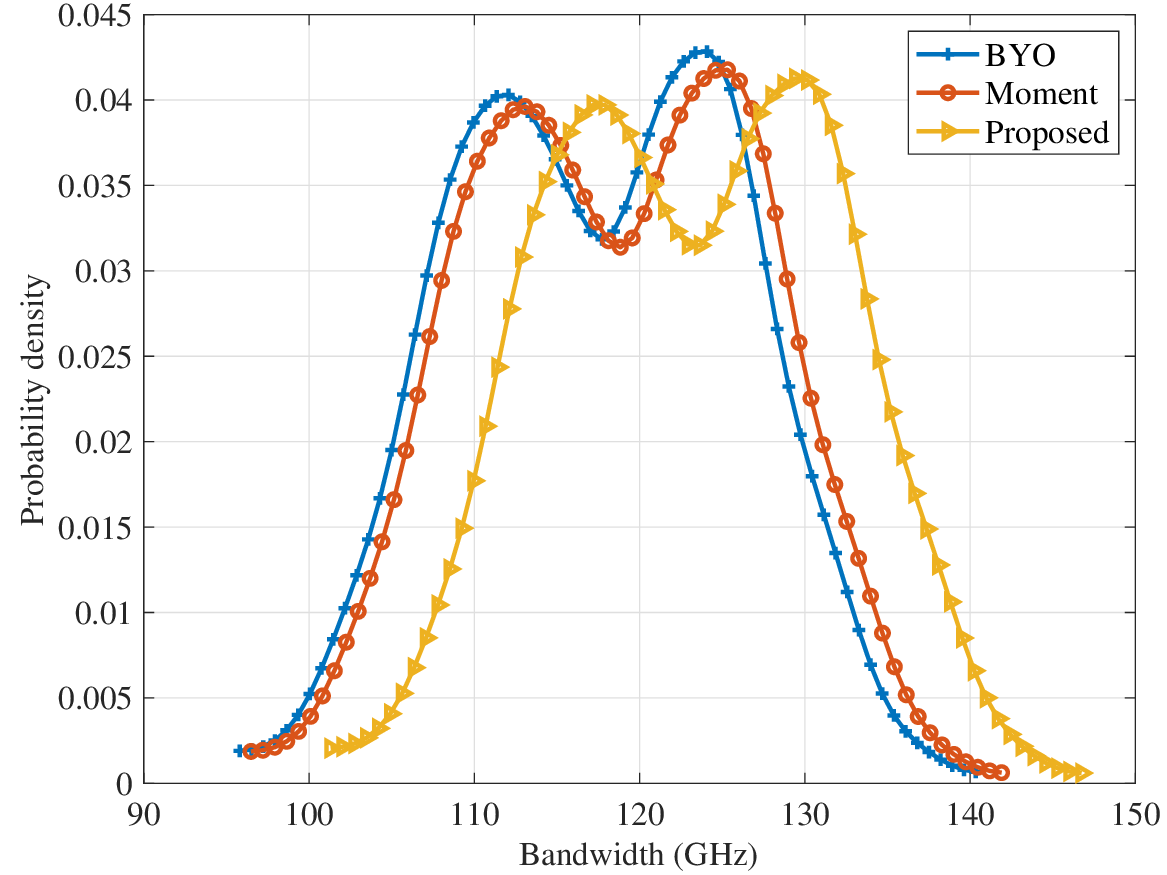}
    \caption{The probability density function of the optimized bandwidth of the microring filter by Bayesian yield optimization~\cite{wang2017efficient}, moment bounding~\cite{cui2020chance} and the proposed PoBO (with $\epsilon=0.1$).}
    \label{fig:Microring_BW_pdf}
\end{figure}

\subsection{Mach-Zehnder Interferometer}
\label{sec:MZI_case}
We consider a third-order Mach-Zehnder interferometer (MZI) which consists of three port coupling and two arms, as shown in Fig.~\ref{fig:benchmark_MZI}.
The coupling coefficients $\tau$ between the MZ arms play an important role in the design, whose relationship with the the gap $g$ (nm) is $\tau = \exp (-\frac{g}{260})$. 
The gap variables $\mat{x}=[g_1, g_2, g_3]$ have the design space of $\mat{X}={[100, 300]}^3$. Their corresponding process variations $\vecpar$ follows a truncated Gaussian mixture distribution (see Appendix~\ref{sec:benchmark_setup}).
We aim to maximize the expected 3-dB bandwidth (BW, in GHz) with probability constraints on the crosstalk (XT, in dB) and the attenuation ($\alpha$, in dB) of the peak transmission. Therefore, the yield-aware chance-constrained design is formulated as
\begin{equation}
\label{eq:MZ_formula}
    \begin{aligned}
        \max_{\mat{x} \in \mat{X}} \quad & \E_{\vecpar}[\text{BW}(\mat{x},\vecpar)]\\
        \st \quad & \P_{\vecpar}(\text{XT}(\mat{x},\vecpar)\le \text{XT}_0) \ge 1-\epsilon_1, \\
        & \P_{\vecpar}(\alpha(\mat{x},\vecpar)\le \alpha_0) \ge 1 -\epsilon_2.
    \end{aligned}
\end{equation}
The two risk levels are set to be equal $\epsilon_1 = \epsilon_2 = \epsilon$. The thresholds of the crosstalk ($\text{XT}_0$) and the attenuation ($\alpha_0$) are -4 dB and 1.6 dB, respectively.

We build three 2nd-order polynomial surrogate models for BW, XT, and $\alpha$, respectively. We further bound the probabilistic yield constraints via an order-5 optimal polynomial kinship function. 
The optimized results and comparisons are listed in Table~\ref{tab:MZI_result}. 
It shows that at the same risk level, the proposed PoBO method can achieve larger bandwidth while meeting the yield requirements and having smaller gaps for the chance constraints. 
The simulation samples are the ones used for building surrogate models.
We list the number of samples to reveal the simulation cost since the simulation time per sample may vary from seconds to hours, which depends on the problem size (small circuits or large circuit) and simulator types (e.g., circuit-level simulation or EM-based PDE simulator).
The proposed PoBO requires the same number of simulations as the moment bounding method~\cite{cui2020chance}, and both of them require much fewer simulation samples than the Bayesian yield optimization due to the efficient surrogate modeling. 
Regarding the CPU time of solving the design optimization problem~\eqref{eq:MZ_formula},  the proposed method takes 15.28 s, 18.19 s, and 15.2 s for the three risk levels, respectively.
The moment method takes 0.39 s, 2.11 s, and 2.21 s, respectively. 
Our Kinship-based optimization is slower than the moment-based method because our method uses higher-order polynomials to bound the probabilistic constraints. However, the optimization overhead is negligible compared with the sample simulation time (especially when a PDE-based simulator is employed).
The BYO takes even less than 0.1 s in the optimization steps, but it requires a huge number of simulation samples, causing a much larger overall CPU time than chance-constrained optimization.
Fig.~\ref{fig:MZI_opt_result} compares the frequency response before and after the yield-aware optimization with $\epsilon=0.1$. 
Our PoBO method has a higher expected bandwidth compared with the Bayesian yield optimization, the moment bounding method, and the initial design.
Fig.~\ref{fig:MZI_BW_pdf} further shows the probability density of the optimized bandwidth by different models. It clearly shows that our proposed method produces the highest bandwidth while meeting the yield requirement. 
\subsection{Microring Add-Drop Filter}
\label{sec:Microring_case}
We further consider the design of an optical add-drop filter consisting of three identical silicon microrings coupled in series, as shown in Fig.~\ref{fig:benchmark_Mircoring}. 
The design variables are the coupling coefficients $\mat{x}=[K_1,K_2,K_3,K_4]$ that are to be optimized within the interval of $\mat{X}={[0.3, 0.6]}^4$. The process variations $\vecpar$ are described by a truncated Gaussian mixture model (see Appendix~\ref{sec:benchmark_setup}).
The design problem is to maximize the expected 3-dB bandwidth (BW, in GHz) with  constraints on the extinction ratio (RE, in dB) of the transmission at the drop port and the roughness ($\sigma_{\text{pass}}$, in dB) of the passband that takes its standard deviation, formulated as:
\begin{equation}
\label{eq:ring_formula}
    \begin{aligned}
        \max_{\mat{x} \in \mat{X}} \quad & \E_{\vecpar}[\text{BW}(\mat{x},\vecpar)]\\
        \st \quad & \P_{\vecpar}(\text{RE}(\mat{x},\vecpar)\ge \text{RE}_0) \ge 1-\epsilon_1, \\
        & \P_{\vecpar}(\sigma_{\text{pass}}(\mat{x},\vecpar)\le \sigma_0) \ge 1 -\epsilon_2.
    \end{aligned}
\end{equation}
The two risk levels are set to be equal $\epsilon_1 = \epsilon_2 = \epsilon$. The thresholds of the extinction ratio ($\text{RE}_0$) and the roughness
of the passband ($\sigma_0$) are 20 dB and 0.65 dB, respectively. 

\begin{figure*}[t!]
    \centering
    \includegraphics[width = 6.6in]{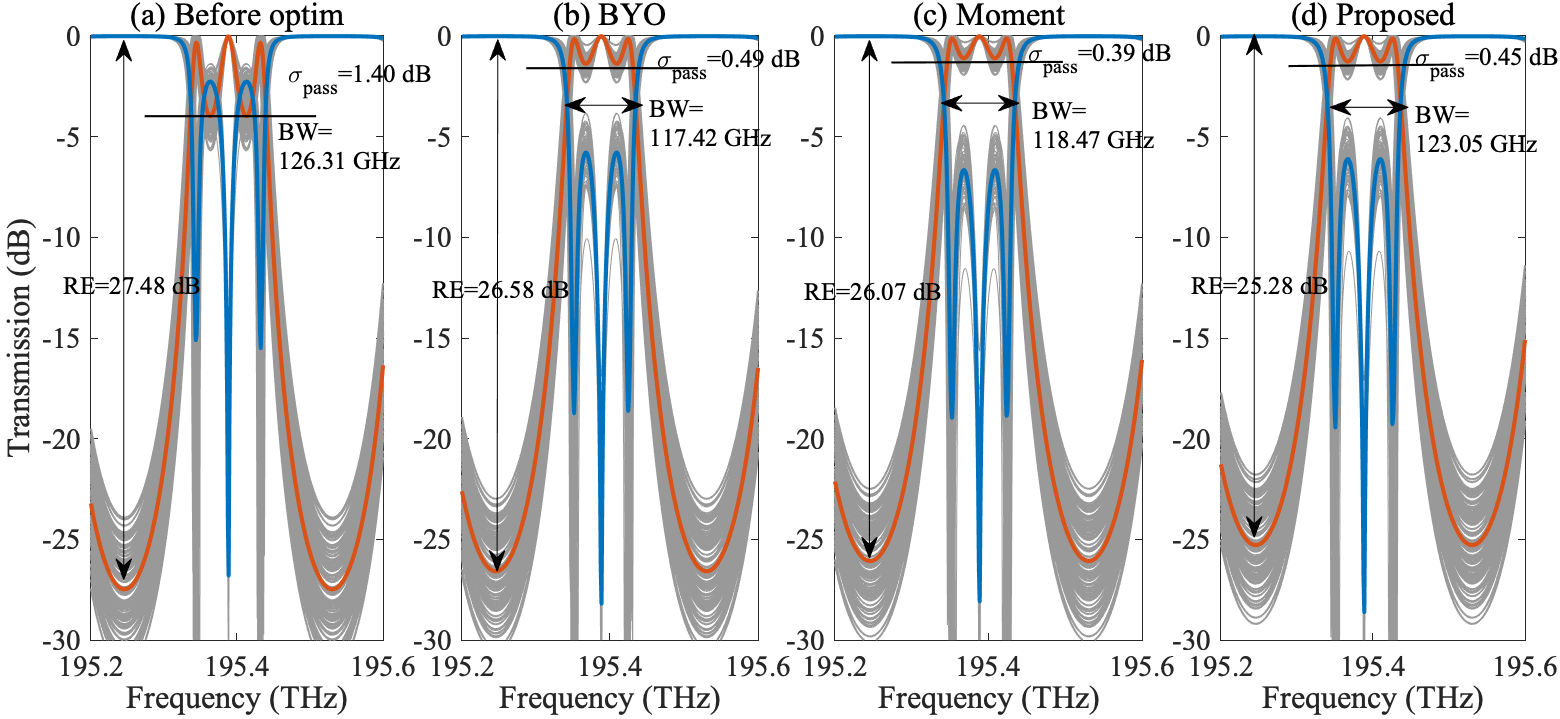}
    \caption{The transmission curves of the microring add-drop filter. The grey lines show the performance uncertainties. The orange and blue curves show the transmission rates at the drop and through ports, respectively. 
    The mean values of the extinction ratio, bandwidth, and roughness are denoted as RE, BW, and $\sigma_{\text{pass}}$, respectively. 
    (a) The initial design (infeasible): $\mat{x}$=[0.45, 0.45, 0.45, 0.45]; (b) Design after Bayesian yield optimization~\cite{wang2017efficient}: $\mat{x}$=[0.5758, 0.3718, 0.3720, 0.5746]; 
    (c) Design with the moment-bounding yield-aware optimization~\cite{cui2020chance}: $\mat{x}$=[0.6, 0.3751, 0.3642 0.6]; 
    (d) Design with the proposed PoBO optimization: $\mat{x}$=[0.6, 0.3971, 0.3642, 0.6].
    }
    \label{fig:Microring_opt_result}
\end{figure*}

\begin{table*}[t]
\centering
\caption{Optimization Results for Microring Add-drop Filter benchmark}
\label{tab:Microring_result}
\begin{threeparttable}
\begin{tabular}{ccccccc}
\toprule
Risk level $\epsilon$  & Method  & $\E_{\vecpar}[\text{BW}]$ (GHz) & $\Delta_1$ (\%) & $\Delta_2$ (\%) & Yield (\%) & Simulation \# \\
\midrule
\multirow{2}{*}{0.05} 
 & Moment~\cite{cui2020chance}   &N/A\tnote{*}      &N/A\tnote{*}  &N/A\tnote{*} &N/A\tnote{*}   &  65 \\
&\textbf{Proposed}           &\textbf{116.85}     &\textbf{5.26}  &\textbf{4.84}  &99.6  & 65 \\ 
\hline
\multirow{2}{*}{0.07} 
&Moment~\cite{cui2020chance}    &112.64    &7.53   &7.42  & 99.9 &  65  \\
&\textbf{Proposed}          &\textbf{120.05}   &\textbf{7.53} & \textbf{6.67}   &99.2 & 65   \\
\hline
\multirow{2}{*}{0.1} 
&Moment~\cite{cui2020chance}     &118.47     &11.11   &10.78  & 99.7  &  65 \\ 
&\textbf{Proposed}          &\textbf{123.05}  &\textbf{11.11} & \textbf{8.33}   &97.5   & 65 \\ 
\hline
N/A & BYO~\cite{wang2017efficient}    &117.42      & N/A  &N/A & 95.1 & 2020 \\
\bottomrule
\end{tabular}
\begin{tablenotes}
   \item[*] The algorithm fails with no feasible solution.                
  \end{tablenotes}
\end{threeparttable}
\end{table*}
Similarly, we build three 2nd-order polynomial surrogate models for BW, RE, and $\sigma_{\text{pass}}$ and bound the chance constraints via an order-5 optimal polynomial kinship function.
The optimized results and comparisons are shown in Table~\ref{tab:Microring_result}. 
The moment bounding~\cite{cui2020chance} fails when $\epsilon=0.05$ since no feasible solution is found under its over-conservative bounding. 
For three risk levels, the optimization time for the moment method is N/A (no feasible solution), 5.89 s, 5.5 s, respectively. The proposed method takes 100.58 s, 110.58 s, and 103.74 s, respectively, but the overhead is negligible compared with the simulation cost. The BYO takes less than 0.1 s in the optimization step, but it takes the most overall time due to the high cost of simulating many samples.
At all risk levels, the proposed method can achieve larger bandwidth while meeting the yield requirements and having smaller gaps for the chance constraints.  
As shown in Fig.~\ref{fig:Microring_BW_pdf}, the proposed PoBO has a higher expected bandwidth compared to Bayesian yield optimization and existing yield-aware chance-constrained optimization via moment bounding~\cite{cui2020chance}.  
Fig.~\ref{fig:Microring_opt_result} shows the frequency response before and after the yield-aware optimization with $\epsilon=0.1$.
For this microring filter benchmark, we further consider a special case where design objective and constraints are the same quantity. For this case, our method still outperforms others (see the details in Appendix~\ref{sec:same_yf}).

\section{Conclusion and Future Work}
This paper has proposed a novel \textbf{Po}lynomial \textbf{B}ounding method for chance-constrained yield-aware \textbf{O}ptimization (PoBO) of photonic ICs with truncated non-Gaussian correlated uncertainties. 
In PoBO, we first construct surrogate models with a few simulation samples for the interested quantities based on available uncertainty quantification solvers. 
To avoid over-conservative design, we have proposed an optimal polynomial kinship function to tightly bound the chance constraints.   
This bounding method can be efficiently implemented without additional simulations.
It also preserves the polynomial form and enables seeking a globally optimal design.
The proposed PoBO is verified with a synthetic function, a Mach-Zehnder interferometer, and a microring add-drop filter. In all experiments, the proposed PoBO has achieved the yield requirements, produced tighter bounds on the chance constraints than the state-of-the-art moment bounding method, and led to better design objective performances with a few simulation samples. On the two photonic IC examples, the proposed method has also reduced the simulation samples by $58\times$ and $31\times$ compared with Bayesian yield optimization.

The theoretical and numerical results of this work have laid the foundation of many future topics. Possible extensions of this work include, but are not limited to: (1) improved algorithms to handle many design parameters and process variation, (2) formulations and algorithms to handle joint chance constraints for yield descriptions, (3) PoBO with non-polynomial surrogates. The proposed framework is very generic, and it can also be employed in other applications beyond EDA, including probabilistic control of energy systems, safety-critical control of autonomous systems, and so forth. 

\appendices

\section{Solution to Optimal Polynomial Kinship~\eqref{eq:poly_kinship}}
\label{sec:optimal_poly}
Given a $\rho$-order optimal kinship function $\kappa_{\rho}(\cdot)$, we introduce two positive semidefinite matrices $\mat{Y}_1 \in \R^{(n_1+1)\times(n_1+1)}$ and $\mat{Y}_2 \in \R^{(n_2+1)\times(n_2+1)}$ with $n_1=\lfloor (\rho-1)/2 \rfloor$ and $n_2=\lfloor (\rho-2)/2 \rfloor$. We further define two series of Hankel matrices $\mat{H}_{1,m} \in \R^{(n_1+1)\times(n_1+1)}$ and $\mat{H}_{2,m} \in \R^{(n_2+1)\times(n_2+1)}$ as
\begin{equation}
    \mat{H}_{k,m}\left(i,j\right)= \begin{cases}
    1, & i+j=m+1 \\
    0, & \text{otherwise}
    \end{cases}, k = 1, 2.
\end{equation}
Based on the sum-of-square representation of a nonnegative univariate polynomial, we can reformulate~\eqref{eq:poly_kinship} as a finite dimensional semidefinite programming~\eqref{eq:sol_poly_kinship}, which can be handled by many efficient solvers and toolboxes~\cite{YALMIP}. The detailed proof will be similar to the Corollary 1 of~\cite{feng2010kinship}, where the difference is the order of nonnegative polynomial.
\begin{equation}\label{eq:sol_poly_kinship}
\centering
\begin{array}{c}
\min_{{\zeta_0},\ldots,{\zeta_\rho}, {\mat{Y}_1}, {\mat{Y}_2}} \; \sum_{i=0}^\rho \frac{(-1)^i}{i+1}\zeta_i\\
    \st \; \zeta_0 = 1, \\
    \sum_{i=0}^\rho (-1)^i \zeta_i = 0, \\
    \text{Tr}(\mat{Y}_1 \mat{H}_{1,m}) + \text{Tr}(\mat{Y}_2 \mat{H}_{2,m}) = \\
    \sum_{i=m+1}^\rho \frac{i!(-1)^{i-m-1}}{k!(i-m-1)!}\zeta_i,
    m=0,1,\ldots,\rho-2, \\
    \mat{Y}_k \succeq 0, \; k=1,2.
\end{array}
\end{equation}
We list some examples of the solved polynomial as below: for $\rho=5$, the polynomial coefficients are ${[1, 7.87, 24.62, 37.49, 27.74, 8.00]}$. 
For $\rho=8$, the polynomial coefficients are ${[1, 13.4, 75, 223.71, 384.3, 381.57, 203.57, 45.19]}$.

\section{Details about Benchmark Setup}
\label{sec:benchmark_setup}
In the MZI benchmark (Sec.~\ref{sec:MZI_case}), the process variations on the coupling coefficients are described by a truncated Gaussian mixture model with two components:
\begin{equation}\label{eq:process_variation_benchmark}
\mu(\mat{\xi})=  \frac{1}{2}\TN_1(\bar{\boldsymbol{\mu}}_1,\mat{\Sigma}_1,\mat{a}_1,\mat{b}_1) + \frac{1}{2}\TN_2(\bar{\boldsymbol{\mu}}_2,\mat{\Sigma}_2,\mat{a}_2,\mat{b}_2), 
\end{equation}
where $\bar{\boldsymbol{\mu}}_1 = -\bar{\boldsymbol{\mu}}_2 = {[3,3,3]}^T$,  
$\mat{\Sigma}_1=\mat{\Sigma}_2 = 3^2\begin{bmatrix}
1 & 0.4 & 0.1 \\
0.4 & 1 & 0.4 \\
0.1 & 0.4 & 1
\end{bmatrix},
\mat{a}_1=-\mat{b}_2={[-6,-6,-6]}^T, \mat{a}_2=-\mat{b}_1={[-12,-12,-12]}^T.
$ 

In the microring benchmark (Sec.~\ref{sec:Microring_case}), the process variations on the coupling coefficients are described the same as Eq.~\eqref{eq:process_variation_benchmark} with different parameters
$\bar{\boldsymbol{\mu}}_1 = -\bar{\boldsymbol{\mu}}_2 = {[0.03,0.03,0.03,0.03]}^T$,
$\mat{\Sigma}_1 = \mat{\Sigma}_2 = 0.03^2\begin{bmatrix}
1 & 0.4 & 0.1 & 0.4\\
0.4 & 1 & 0.4 &0.1\\
0.1 & 0.4 & 1 &0.4\\
0.4 & 0.1 &0.4 &1
\end{bmatrix},
\mat{a}_1=-\mat{b}_2={[-0.06,-0.06,-0.06,-0.06]}^T, \mat{a}_2=-\mat{b}_1={[-0.12,-0.12,-0.12,-0.12]}^T.$ 

\section{Additional Case Study}
\label{sec:same_yf}
\begin{table}[t]
\centering
\caption{Bandwidth-Constrained Optimization Results for Microring Add-drop Filter}
\label{tab:Microring_result_sameyf}
\begin{adjustbox}{width=3.3in}
\begin{threeparttable}
\begin{tabular}{cccccc}
\toprule
Risk level $\epsilon$  & Method  & $\E_{\vecpar}[\text{BW}]$ (GHz) & $\Delta$ (\%)  & Yield (\%) & Simulation \# \\
\midrule
\multirow{2}{*}{0.05} 
 & Moment~\cite{cui2020chance}   &N/A\tnote{*}        &N/A\tnote{*} &N/A\tnote{*}   &  65 \\
&\textbf{Proposed}           &\textbf{105.66}     &\textbf{3.89}   &98.7  & 65  \\
\hline
\multirow{2}{*}{0.07} 
&Moment~\cite{cui2020chance}    & 93.96    & 7.53   & 100 &  65  \\
&\textbf{Proposed}          &\textbf{107.28}   &\textbf{4.09}    &96.8 & 65  \\ 
\hline
\multirow{2}{*}{0.1} 
&Moment~\cite{cui2020chance}     & 98.88     & 11.11     & 100  &  65 \\ 
&\textbf{Proposed}          &\textbf{109.01}  &\textbf{3.67}    &93.3   & 65 \\ 
\hline
N/A & BYO~\cite{wang2017efficient}    &102.05      & N/A  & 99.8 &  2020 \\
\bottomrule
\end{tabular}
\begin{tablenotes}
   \item[*] The algorithm fails with no feasible solution.                
  \end{tablenotes}
\end{threeparttable}
\end{adjustbox}
\end{table}

In this case study, for the same microring add-drop filter in section~\ref{sec:Microring_case}, 
we aim to optimize its bandwidth while holding a constraint on the bandwidth as well. The yield-aware chance-constrained design is formulated as
\begin{equation}
    \begin{aligned}
        \max_{\mat{x} \in \mat{X}} \quad & \E_{\vecpar}[\text{BW}(\mat{x},\vecpar)]\\
        \st \quad & \P_{\vecpar}(\text{BW}(\mat{x},\vecpar) \le \text{BW}_0) \ge 1-\epsilon,
    \end{aligned}
\end{equation}
where the bandwidth threshold $\text{BW}_0$ is 120 GHz. The design solutions and comparisons are summarized in Table~\ref{tab:Microring_result_sameyf}. Similar to the cases in section~\ref{sec:numerical_results}, our method still works when the objective and constraint are the same quantity and outperforms the other approaches.

\small{
\bibliographystyle{Bib/IEEEtran}
\bibliography{Bib/reference}

\begin{thebibliography}{10}
\providecommand{\url}[1]{#1}
\csname url@samestyle\endcsname
\providecommand{\newblock}{\relax}
\providecommand{\bibinfo}[2]{#2}
\providecommand{\BIBentrySTDinterwordspacing}{\spaceskip=0pt\relax}
\providecommand{\BIBentryALTinterwordstretchfactor}{4}
\providecommand{\BIBentryALTinterwordspacing}{\spaceskip=\fontdimen2\font plus
\BIBentryALTinterwordstretchfactor\fontdimen3\font minus
  \fontdimen4\font\relax}
\providecommand{\BIBforeignlanguage}[2]{{%
\expandafter\ifx\csname l@#1\endcsname\relax
\typeout{** WARNING: IEEEtran.bst: No hyphenation pattern has been}%
\typeout{** loaded for the language `#1'. Using the pattern for}%
\typeout{** the default language instead.}%
\else
\language=\csname l@#1\endcsname
\fi
#2}}
\providecommand{\BIBdecl}{\relax}
\BIBdecl

\bibitem{gielen2008emerging}
G.~Gielen, P.~De~Wit, E.~Maricau, J.~Loeckx, J.~Martin-Martinez, B.~Kaczer,
  G.~Groeseneken, R.~Rodriguez, and M.~Nafria, ``Emerging yield and reliability
  challenges in nanometer {CMOS} technologies,'' in \emph{Proc. Design, Autom.
  Test Eur. Conf. Exhibit.}, 2008, pp. 1322--1327.

\bibitem{chen2013process}
X.~Chen, M.~Mohamed, Z.~Li, L.~Shang, and A.~R. Mickelson, ``Process variation
  in silicon photonic devices,'' \emph{Appl. Opt.}, vol.~52, no.~31, pp.
  7638--7647, 2013.

\bibitem{lipka2016systematic}
T.~Lipka, J.~M{\"u}ller, and H.~K. Trieu, ``Systematic nonuniformity analysis
  of amorphous silicon-on-insulator photonic microring resonators,'' \emph{J.
  Light. Technol.}, vol.~34, no.~13, pp. 3163--3170, 2016.

\bibitem{bogaerts2019layout}
W.~Bogaerts, Y.~Xing, and U.~Khan, ``Layout-aware variability analysis, yield
  prediction, and optimization in photonic integrated circuits,'' \emph{IEEE J.
  Sel. Top. Quantum Electron.}, vol.~25, no.~5, pp. 1--13, 2019.

\bibitem{antreich1994circuit}
K.~J. Antreich, H.~E. Graeb, and C.~U. Wieser, ``Circuit analysis and
  optimization driven by worst-case distances,'' \emph{IEEE Trans.
  Comput.-Aided Design Integr. Circuits Syst.}, vol.~13, no.~1, pp. 57--71,
  1994.

\bibitem{gong2014variability}
F.~Gong, Y.~Shi, H.~Yu, and L.~He, ``Variability-aware parametric yield
  estimation for analog/mixed-signal circuits: Concepts, algorithms, and
  challenges,'' \emph{IEEE Des. Test}, vol.~31, no.~4, pp. 6--15, 2014.

\bibitem{lu2017performance}
Z.~Lu, J.~Jhoja, J.~Klein, X.~Wang, A.~Liu, J.~Flueckiger, J.~Pond, and
  L.~Chrostowski, ``Performance prediction for silicon photonics integrated
  circuits with layout-dependent correlated manufacturing variability,''
  \emph{Opt. Express}, vol.~25, no.~9, pp. 9712--9733, 2017.

\bibitem{pond2017predicting}
J.~Pond, J.~Klein, J.~Fl{\"u}ckiger, X.~Wang, Z.~Lu, J.~Jhoja, and
  L.~Chrostowski, ``Predicting the yield of photonic integrated circuits using
  statistical compact modeling,'' in \emph{Integrated Optics: Physics and
  Simulations III}, vol. 10242, 2017, p. 102420S.

\bibitem{xu2009regular}
Y.~Xu, K.-L. Hsiung, X.~Li, L.~T. Pileggi, and S.~P. Boyd, ``Regular
  analog/{RF} integrated circuits design using optimization with recourse
  including ellipsoidal uncertainty,'' \emph{IEEE Trans. Comput.-Aided Design
  Integr. Circuits Syst.}, vol.~28, no.~5, pp. 623--637, 2009.

\bibitem{yu2007yield}
G.~Yu and P.~Li, ``Yield-aware analog integrated circuit optimization using
  geostatistics motivated performance modeling,'' in \emph{Proc. Intl. Conf.
  Computer Aided Design}, 2007, pp. 464--469.

\bibitem{tiwary2006generation}
S.~K. Tiwary, P.~K. Tiwary, and R.~A. Rutenbar, ``Generation of yield-aware
  {Pareto} surfaces for hierarchical circuit design space exploration,'' in
  \emph{Proc. Design Autom. Conf}, 2006, pp. 31--36.

\bibitem{li2009yield}
Y.~Li and V.~Stojanovic, ``Yield-driven iterative robust circuit optimization
  algorithm,'' in \emph{Proc. Design Autom. Conf}, 2009, pp. 599--604.

\bibitem{liu2011efficient}
B.~Liu, F.~V. Fern{\'a}ndez, and G.~G. Gielen, ``Efficient and accurate
  statistical analog yield optimization and variation-aware circuit sizing
  based on computational intelligence techniques,'' \emph{IEEE Trans.
  Comput.-Aided Design Integr. Circuits Syst.}, vol.~30, no.~6, pp. 793--805,
  2011.

\bibitem{barros2010analog}
M.~Barros, J.~Guilherme, and N.~Horta, ``Analog circuits optimization based on
  evolutionary computation techniques,'' \emph{Integration}, vol.~43, no.~1,
  pp. 136--155, 2010.

\bibitem{wang2017efficient}
M.~Wang, F.~Yang, C.~Yan, X.~Zeng, and X.~Hu, ``Efficient {Bayesian} yield
  optimization approach for analog and {SRAM} circuits,'' in \emph{Proc. Design
  Autom. Conf}, 2017, pp. 1--6.

\bibitem{singhee2008practical}
A.~Singhee, S.~Singhal, and R.~A. Rutenbar, ``Practical, fast {Monte Carlo}
  statistical static timing analysis: Why and how,'' in \emph{Proc. Intl. Conf.
  Computer Aided Design}, 2008, pp. 190--195.

\bibitem{papoulis2001probability}
A.~Papoulis and H.~Saunders, \emph{Probability, random variables and stochastic
  processes}.\hskip 1em plus 0.5em minus 0.4em\relax McGraw-Hill, 2001.

\bibitem{gu2008efficient}
C.~Gu and J.~Roychowdhury, ``An efficient, fully nonlinear, variability-aware
  non-{Monte-Carlo} yield estimation procedure with applications to {SRAM}
  cells and ring oscillators,'' in \emph{Proc. Asia South Pac. Design Autom.
  Conf.}, 2008, pp. 754--761.

\bibitem{gong2012fast}
F.~Gong, X.~Liu, H.~Yu, S.~X. Tan, J.~Ren, and L.~He, ``A fast
  non-{Monte-Carlo} yield analysis and optimization by stochastic orthogonal
  polynomials,'' \emph{ACM Trans. Des. Autom. Electron. Syst.}, vol.~17, no.~1,
  pp. 1--23, 2012.

\bibitem{9428031}
Z.~Gao and R.~Rohrer, ``Efficient non-{Monte-Carlo} yield estimation,''
  \emph{IEEE Trans. Comput.-Aided Design Integr. Circuits Syst.}, pp. 1--1,
  2021.

\bibitem{shi2019meta}
X.~Shi, H.~Yan, Q.~Huang, J.~Zhang, L.~Shi, and L.~He, ``Meta-model based
  high-dimensional yield analysis using low-rank tensor approximation,'' in
  \emph{Proc. Design Autom. Conf}, 2019, pp. 1--6.

\bibitem{yao2014efficient}
J.~Yao, Z.~Ye, and Y.~Wang, ``An efficient {SRAM} yield analysis and
  optimization method with adaptive online surrogate modeling,'' \emph{IEEE
  Trans. Very Large Scale Integr. (VLSI) Syst.}, vol.~23, no.~7, pp.
  1245--1253, 2014.

\bibitem{li2006asymptotic}
X.~Li, J.~Le, P.~Gopalakrishnan, and L.~T. Pileggi, ``Asymptotic probability
  extraction for nonnormal performance distributions,'' \emph{IEEE Trans.
  Comput.-Aided Design Integr. Circuits Syst.}, vol.~26, no.~1, pp. 16--37,
  2006.

\bibitem{li2004robust}
X.~Li, P.~Gopalakrishnan, Y.~Xu, and T.~Pileggi, ``Robust analog/{RF} circuit
  design with projection-based posynomial modeling,'' in \emph{Proc. Intl.
  Conf. Computer Aided Design}, 2004, pp. 855--862.

\bibitem{li2008quadratic}
X.~Li, Y.~Zhan, and L.~T. Pileggi, ``Quadratic statistical $ max $
  approximation for parametric yield estimation of analog/rf integrated
  circuits,'' \emph{IEEE Trans. Comput.-Aided Design Integr. Circuits Syst.},
  vol.~27, no.~5, pp. 831--843, 2008.

\bibitem{ciccazzo2015svm}
A.~Ciccazzo, G.~Di~Pillo, and V.~Latorre, ``A {SVM} surrogate model-based
  method for parametric yield optimization,'' \emph{IEEE Trans. Comput.-Aided
  Design Integr. Circuits Syst.}, vol.~35, no.~7, pp. 1224--1228, 2015.

\bibitem{ma2020support}
H.~Ma, E.-P. Li, A.~C. Cangellaris, and X.~Chen, ``Support vector
  regression-based active subspace {(SVR-AS)} modeling of high-speed links for
  fast and accurate sensitivity analysis,'' \emph{IEEE Access}, vol.~8, pp.
  74\,339--74\,348, 2020.

\bibitem{sanabria2020gaussian}
A.~C. Sanabria-Borb{\'o}n, S.~Soto-Aguilar, J.~J. Estrada-L{\'o}pez,
  D.~Allaire, and E.~S{\'a}nchez-Sinencio, ``Gaussian-process-based surrogate
  for optimization-aided and process-variations-aware analog circuit design,''
  \emph{Electronics}, vol.~9, no.~4, p. 685, 2020.

\bibitem{wang2017yield}
M.~Wang, W.~Lv, F.~Yang, C.~Yan, W.~Cai, D.~Zhou, and X.~Zeng, ``Efficient
  yield optimization for analog and {SRAM} circuits via {Gaussian} process
  regression and adaptive yield estimation,'' \emph{IEEE Trans. Comput.-Aided
  Design Integr. Circuits Syst.}, vol.~37, no.~10, pp. 1929--1942, 2017.

\bibitem{wang2014enabling}
Y.~Wang, M.~Orshansky, and C.~Caramanis, ``Enabling efficient analog synthesis
  by coupling sparse regression and polynomial optimization,'' in \emph{Proc.
  Design Autom. Conf}, 2014, pp. 1--6.

\bibitem{xiu2003modeling}
D.~Xiu and G.~E. Karniadakis, ``Modeling uncertainty in flow simulations via
  generalized polynomial chaos,'' \emph{J. Comput. Phys.}, vol. 187, no.~1, pp.
  137--167, 2003.

\bibitem{zhang2013stochastic}
Z.~Zhang, T.~A. El-Moselhy, I.~M. Elfadel, and L.~Daniel, ``Stochastic testing
  method for transistor-level uncertainty quantification based on generalized
  polynomial chaos,'' \emph{IEEE Trans. Comput.-Aided Design Integr. Circuits
  Syst.}, vol.~32, no.~10, pp. 1533--1545, 2013.

\bibitem{trinchero2020combining}
R.~Trinchero and F.~G. Canavero, ``Combining {LS-SVM} and {GP} regression for
  the uncertainty quantification of the {EMI} of power converters affected by
  several uncertain parameters,'' \emph{IEEE Trans. Electromagn. Compat.},
  vol.~62, no.~5, pp. 1755--1762, 2020.

\bibitem{li2010finding}
X.~Li, ``Finding deterministic solution from underdetermined equation:
  large-scale performance variability modeling of analog/{RF} circuits,''
  \emph{IEEE Trans. Comput.-Aided Design Integr. Circuits Syst.}, vol.~29,
  no.~11, pp. 1661--1668, 2010.

\bibitem{zhang2014enabling}
Z.~Zhang, X.~Yang, I.~V. Oseledets, G.~E. Karniadakis, and L.~Daniel,
  ``Enabling high-dimensional hierarchical uncertainty quantification by
  {ANOVA} and tensor-train decomposition,'' \emph{IEEE Trans. Comput.-Aided
  Design Integr. Circuits Syst.}, vol.~34, no.~1, pp. 63--76, 2014.

\bibitem{zhang2016big}
Z.~Zhang, T.-W. Weng, and L.~Daniel, ``Big-data tensor recovery for
  high-dimensional uncertainty quantification of process variations,''
  \emph{IEEE Trans. Compon. Packag. Manuf. Technol.}, vol.~7, no.~5, pp.
  687--697, 2016.

\bibitem{He2020EPEPS}
Z.~{He} and Z.~{Zhang}, ``High-dimensional uncertainty quantification via
  active and rank-adaptive tensor regression,'' in \emph{Proc. Electr. Perform.
  Electron. Packag. Syst.}, 2020, pp. 1--3.

\bibitem{he2021high}
Z.~He and Z.~Zhang, ``High-dimensional uncertainty quantification via tensor
  regression with rank determination and adaptive sampling,'' \emph{IEEE Trans.
  Compon. Packag. Manuf. Technol.}, vol.~11, no.~9, pp. 1317--1328, 2021.

\bibitem{He2019ICCAD}
Z.~{He}, W.~{Cui}, C.~{Cui}, T.~{Sherwood}, and Z.~{Zhang}, ``Efficient
  uncertainty modeling for system design via mixed integer programming,'' in
  \emph{Proc. Intl. Conf. Computer Aided Design}, 2019, pp. 1--8.

\bibitem{cui2018stochastic}
C.~Cui and Z.~Zhang, ``Stochastic collocation with non-{G}aussian correlated
  process variations: Theory, algorithms, and applications,'' \emph{IEEE Trans.
  Compon. Packag. Manuf. Technol.}, vol.~9, no.~7, pp. 1362--1375, 2018.

\bibitem{manfredi2014stochastic}
P.~Manfredi, D.~V. Ginste, D.~De~Zutter, and F.~G. Canavero, ``Stochastic
  modeling of nonlinear circuits via {SPICE}-compatible spectral equivalents,''
  \emph{IEEE Trans. Circuits Syst. I, Reg. Papers}, vol.~61, no.~7, pp.
  2057--2065, 2014.

\bibitem{ahadi2016sparse}
M.~Ahadi and S.~Roy, ``Sparse linear regression ({SPLINER}) approach for
  efficient multidimensional uncertainty quantification of high-speed
  circuits,'' \emph{IEEE Trans. Comput.-Aided Design Integr. Circuits Syst.},
  vol.~35, no.~10, pp. 1640--1652, 2016.

\bibitem{kaintura2018review}
A.~Kaintura, T.~Dhaene, and D.~Spina, ``Review of polynomial chaos-based
  methods for uncertainty quantification in modern integrated circuits,''
  \emph{Electronics}, vol.~7, no.~3, p.~30, 2018.

\bibitem{wang2016re}
F.~Wang, S.~Yin, M.~Jun, X.~Li, T.~Mukherjee, R.~Negi, and L.~Pileggi,
  ``Re-thinking polynomial optimization: efficient programming of
  reconfigurable radio frequency {(RF)} systems by convexification,'' in
  \emph{Proc. Asia South Pac. Design Autom. Conf.}, 2016, pp. 545--550.

\bibitem{tao2018graph}
J.~Tao, Y.~Su, D.~Zhou, X.~Zeng, and X.~Li, ``Graph-constrained sparse
  performance modeling for analog circuit optimization via {SDP} relaxation,''
  \emph{IEEE Trans. Comput.-Aided Design Integr. Circuits Syst.}, vol.~38,
  no.~8, pp. 1385--1398, 2018.

\bibitem{waqas2021performance}
A.~Waqas, P.~Manfredi, and D.~Melati, ``Performance variability analysis of
  photonic circuits with many correlated parameters,'' \emph{J. Light.
  Technol.}, 2021.

\bibitem{waqas2018stochastic}
A.~Waqas, D.~Melati, P.~Manfredi, and A.~Melloni, ``Stochastic process design
  kits for photonic circuits based on polynomial chaos augmented
  macro-modelling,'' \emph{Opt. Express}, vol.~26, no.~5, pp. 5894--5907, 2018.

\bibitem{weng2015uncertainty}
T.-W. Weng, Z.~Zhang, Z.~Su, Y.~Marzouk, A.~Melloni, and L.~Daniel,
  ``Uncertainty quantification of silicon photonic devices with correlated and
  non-{Gaussian} random parameters,'' \emph{Opt. Express}, vol.~23, no.~4, pp.
  4242--4254, 2015.

\bibitem{weng2017stochastic}
T.-W. Weng, D.~Melati, A.~Melloni, and L.~Daniel, ``Stochastic simulation and
  robust design optimization of integrated photonic filters,''
  \emph{Nanophotonics}, vol.~6, no.~1, pp. 299--308, 2017.

\bibitem{cui2020chance}
C.~Cui, K.~Liu, and Z.~Zhang, ``Chance-constrained and yield-aware optimization
  of photonic {IC}s with non-{G}aussian correlated process variations,''
  \emph{IEEE Trans. Comput.-Aided Design Integr. Circuits Syst.}, 2020.

\bibitem{shapiro2014lectures}
A.~Shapiro, D.~Dentcheva, and A.~Ruszczy{\'n}ski, \emph{Lectures on stochastic
  programming: modeling and theory}.\hskip 1em plus 0.5em minus 0.4em\relax
  SIAM, 2014.

\bibitem{mesbah2014stochastic}
A.~Mesbah, S.~Streif, R.~Findeisen, and R.~D. Braatz, ``Stochastic nonlinear
  model predictive control with probabilistic constraints,'' in \emph{Proc. Am.
  Control Conf.}, 2014, pp. 2413--2419.

\bibitem{vitus2015stochastic}
M.~P. Vitus, Z.~Zhou, and C.~J. Tomlin, ``Stochastic control with uncertain
  parameters via chance constrained control,'' \emph{IEEE Trans. Autom.
  Control}, vol.~61, no.~10, pp. 2892--2905, 2015.

\bibitem{wang2017chance}
Z.~Wang, C.~Shen, F.~Liu, X.~Wu, C.-C. Liu, and F.~Gao, ``Chance-constrained
  economic dispatch with {non-Gaussian} correlated wind power uncertainty,''
  \emph{IEEE Trans. Power Syst.}, vol.~32, no.~6, pp. 4880--4893, 2017.

\bibitem{van2016generalized}
B.~P. Van~Parys, P.~J. Goulart, and D.~Kuhn, ``Generalized {Gauss} inequalities
  via semidefinite programming,'' \emph{Math. Program.}, vol. 156, no. 1-2, pp.
  271--302, 2016.

\bibitem{feng2010kinship}
C.~Feng, F.~Dabbene, and C.~M. Lagoa, ``A kinship function approach to robust
  and probabilistic optimization under polynomial uncertainty,'' \emph{IEEE
  Trans. Autom. Control}, vol.~56, no.~7, pp. 1509--1523, 2010.

\bibitem{calafiore2006distributionally}
G.~C. Calafiore and L.~El~Ghaoui, ``On distributionally robust
  chance-constrained linear programs,'' \emph{J. Optim. Theory Appl.}, vol.
  130, no.~1, pp. 1--22, 2006.

\bibitem{henrion2009gloptipoly}
D.~Henrion, J.-B. Lasserre, and J.~L{\"o}fberg, ``Gloptipoly 3: moments,
  optimization and semidefinite programming,'' \emph{Optim. Methods Softw.},
  vol.~24, no. 4-5, pp. 761--779, 2009.

\bibitem{cui2019high}
C.~Cui and Z.~Zhang, ``High-dimensional uncertainty quantification of
  electronic and photonic ic with non-{Gaussian} correlated process
  variations,'' \emph{IEEE Trans. Comput.-Aided Design Integr. Circuits Syst.},
  vol.~39, no.~8, pp. 1649--1661, 2019.

\bibitem{golub1969calculation}
G.~H. Golub and J.~H. Welsch, ``Calculation of {Gauss} quadrature rules,''
  \emph{Math. Comp.}, vol.~23, no. 106, pp. 221--230, 1969.

\bibitem{gerstner1998numerical}
T.~Gerstner and M.~Griebel, ``Numerical integration using sparse grids,''
  \emph{Numer. Algorithms}, vol.~18, no.~3, pp. 209--232, 1998.

\bibitem{lasserre2001global}
J.~B. Lasserre, ``Global optimization with polynomials and the problem of
  moments,'' \emph{SIAM J. Optim.}, vol.~11, no.~3, pp. 796--817, 2001.

\bibitem{lasserre2008semidefinite}
------, ``A semidefinite programming approach to the generalized problem of
  moments,'' \emph{Math. Program.}, vol. 112, no.~1, pp. 65--92, 2008.

\bibitem{nie2014optimality}
J.~Nie, ``Optimality conditions and finite convergence of lasserre’s
  hierarchy,'' \emph{Math. Program.}, vol. 146, no.~1, pp. 97--121, 2014.

\bibitem{YALMIP}
J.~Lofberg, ``{YALMIP} : a toolbox for modeling and optimization in {MATLAB},''
  in \emph{Proc. Intl. Conf. Robot. Autom.}, 2004, pp. 284--289.

\end{thebibliography}
}

\begin{IEEEbiography}
[{\includegraphics[width=1in,height=1.25in,clip,keepaspectratio]{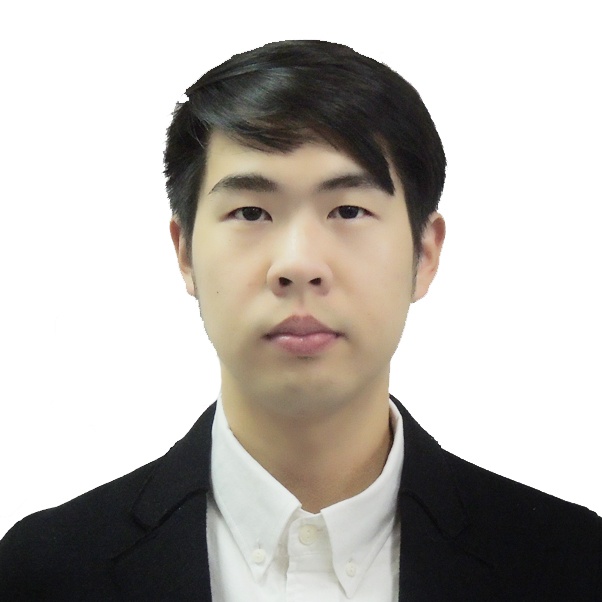}}]{Zichang He} (S’21) received the B.E. degree in Detection, Guidance and Control Technology in 2018 from Northwestern Polytechnical University, Xi'an, China. In 2018 he joined the Department of Electrical and Computer Engineering at University of California, Santa Barbara as a Ph.D. student.

Zichang's research activities are mainly focused on uncertainty quantification and tensor related topics with applications on design automation, machine learning, and quantum computing. He is the recipient of best student paper award in IEEE Electrical Performance of Electronic Packaging and Systems (EPEPS) conference in 2020, the Outstanding Teaching Assistant award in the department of ECE at UCSB in 2020 and 2021, and IEE Excellent in Research Fellowship in 2021.
\end{IEEEbiography}

\begin{IEEEbiography}
 [{\includegraphics[width=1in,height=1.25in,clip,keepaspectratio]{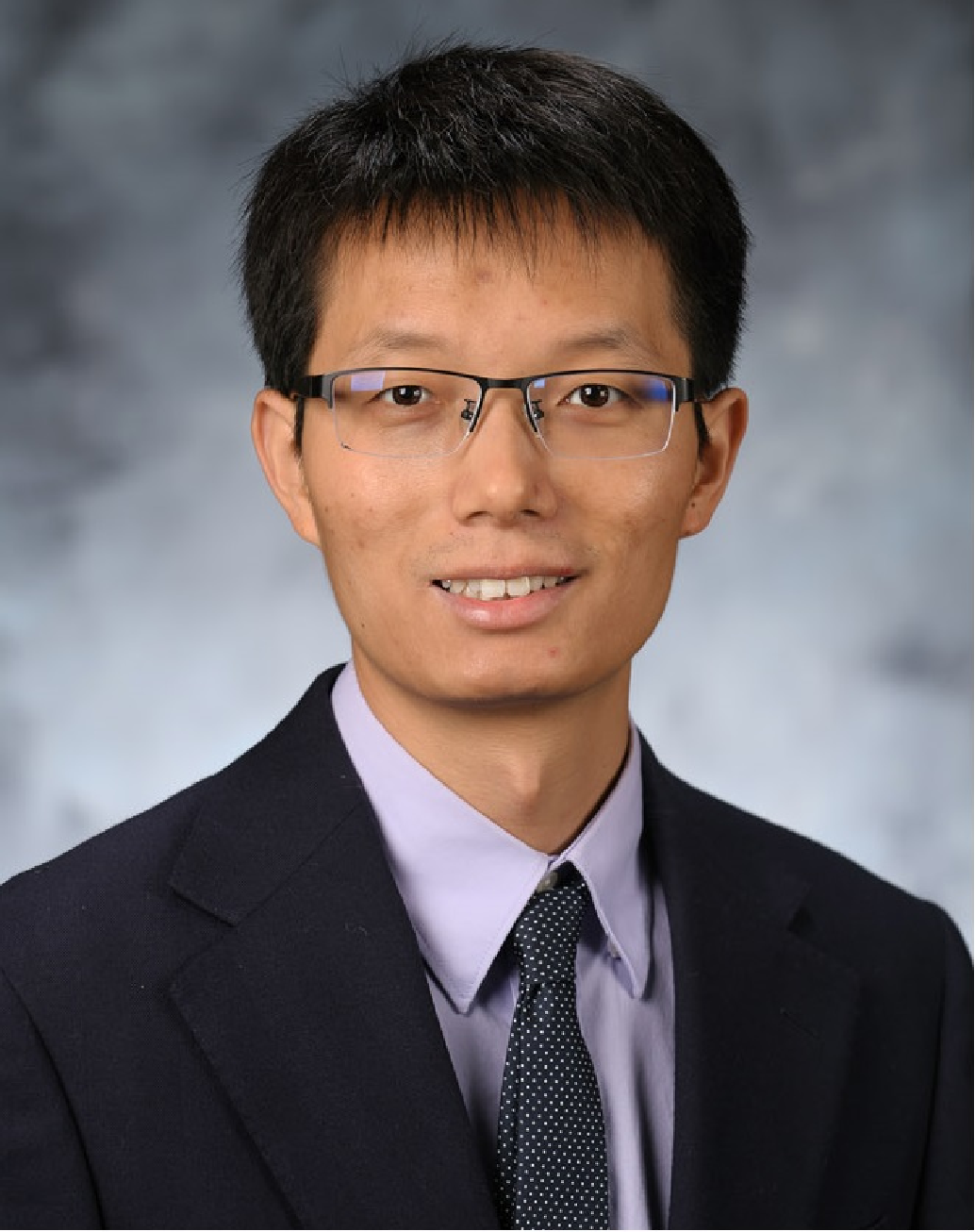}}]{Zheng Zhang} (M'15) received his Ph.D degree in Electrical Engineering and Computer Science from the Massachusetts Institute of Technology (MIT), Cambridge, MA, in 2015. He is an Assistant Professor of Electrical and Computer Engineering with the University of California at Santa Barbara (UCSB), CA. His research interests include uncertainty quantification and tensor computational methods with applications to multi-domain design automation, robust/safe and high-dimensional machine learning and its algorithm/hardware co-design. 

Dr. Zhang received the Best Paper Award of IEEE Transactions on Computer-Aided Design of Integrated Circuits and Systems in 2014, two Best Paper Awards of IEEE Transactions on Components, Packaging and Manufacturing Technology in 2018 and 2020, and three Best Conference Paper Awards (IEEE EPEPS 2018 and 2020, IEEE SPI 2016). His Ph.D. dissertation was recognized by the ACM SIGDA Outstanding Ph.D. Dissertation Award in Electronic Design Automation in 2016, and by the Doctoral Dissertation Seminar Award from the Microsystems Technology Laboratory of MIT in 2015. He received the NSF CAREER Award in 2019, Facebook Research Award in 2020, ACM SIGDA Outstanding New Faculty Award in 2021, and IEEE CEDA Ernest S. Kuh Early Career Award in 2021.
\end{IEEEbiography}

\end{document}